\documentclass[10pt]{amsart}
\usepackage{graphicx}
\usepackage{color}
\usepackage{latexsym}
\usepackage{amssymb}                
\usepackage{amsmath}
\usepackage{amsthm}
\usepackage{amscd}
\usepackage{enumerate}
\usepackage{verbatim}
\usepackage{xy}
\xyoption{all}
\theoremstyle{plain}
\newtheorem{thm}{Theorem}
\newtheorem{lem}{Lemma}
\newtheorem{definition}{Definition}
\newtheorem{conj}{Conjecture}
\newtheorem{defi}{Definition}
\newtheorem{prop}{Proposition}
\newtheorem{cor}{Corollary} 
  
\theoremstyle{remark}
\newtheorem{rem}{Remark}
\newtheorem{question}{Question}
\newtheorem{ex}{Example}
\numberwithin{equation}{section}
\def\ens{\ensuremath}                                     
\def\bydef{\ens{\stackrel{\text{def}}{=}}}
\newcommand\mtb[1]{\ens{\mathbb{#1}}}                     
\newcommand{\LS}{\ensuremath{\underset{n=1}{\overset{\infty}{\cap}} \, {\underset{i=n}{\overset{\infty}{\cup}}}\,}}
\newcommand{\boundsum}{\underset{i=1}{\overset{\infty}{\sum}}}
\newcommand{\N}{\mtb{N}}	   \newcommand{\Q}{\mtb{Q}}   	\newcommand{\R}{\mtb{R}}	
	   \newcommand{\Z}{\mtb{Z}}

\newcommand\mtc[1]{\ens{\mathcal{#1}}}           
\newcommand{\MM}{\mtc{P}}    \newcommand{\RR}{\mtc{R}}       \newcommand{\BB}{\mtc{B}}
    
\newcommand{\alp}{\ens{\alpha}}  
\newcommand{\del}{\ens{\delta}}  \newcommand{\eps}{\ens{\epsilon}}
	\newcommand{\lam}{\ens{\lambda}}
\newcommand\bld[1]{\ens{\boldsymbol{#1}}} 
\newcommand{\bfx}{\ens{\bld x}}    
  
\newcommand{\bfs}{\ens{\bld s}}
\newcommand\uv[2]{\scalebox{#1}{#2}} 
\newcommand\uvm[2]{\scalebox{#1}{\ens{#2}}} 
\newcommand{\card}{{\bf card}}                       
\newcommand{\hdim}{{ \uvm{.95}{\rm{dim_{\uv{.65}{\rm H}}}}}}   
\newcommand{\dist}{{\bf dist}}                       

\newcommand\rbm[2]{\raisebox{-#1 mm}{#2}}            
\newcommand{\intpart}[1]{\big[ #1 \big]}             

\newcommand{\fracpart}[1]{\|#1 \|}                   
\newcommand{\cf}[1] {\big[ #1 \big]_\downarrow}      
\newcommand\hl[1]{{\center \color{red} \ens{\bs\bs\bs} \\}} 
\newcommand\hsp[1]{\mbox{}\hspace{#1mm}} 
\newcommand\hspm[1]{\mbox{}\hspace{-#1mm}} 
\newcommand\vsp[1]{\par \vspace{#1mm}} 
\newcommand\vspm[1]{\par \vspace{-#1mm}} 
\newcommand{\np}{\newpage} 
\def\bs{{\bigstar}}                     
\def\ui{\ens{\mtc I}}                   
\def\rmq{\ens{\R\!\setminus\!\Q}}       
\def\leb{\ens{\lambda}}                 
\def\mms{\mbox{m.\,m.-space}}           
\def\mmps{\mbox{m.\,m.\,p.-system}}     
\newcommand{\ttimes}{\!\times\!}        
\newcommand{\supp}{\uv{.9}{\ens{\mathbf{supp}}}}  
\newcommand{\toto}{^{(2)}}              
\newcommand\mbc[1]{}                                                    
\title[Diophantine properties of IETs \today]
{Diophantine properties of IETs and general systems:\\ Quantitative proximality 
and connectivity}
\author[M. Boshernitzan]{Michael Boshernitzan}
\author[J. Chaika]{Jon Chaika}
\begin{document}
\
\maketitle
\vspm5
\begin{abstract}
We present shrinking targets results for general systems with the emphasis on applications
for IETs (interval exchange transformations)   $(\ui,T)$,  $\ui=[0,1)$. 
In particular, we prove that if an IET  $(\ui,T)$  is ergodic  
(relative to the Lebesgue measure $\lam$),  then the equality
\[
\liminf_{n\to\infty}\limits\, n\,|T^n(x)-y|=0 \tag{A1}
\]
holds for  $\lam\ttimes\lam$-a.\,a.   $(x,y)\in \ui^2$.
The ergodicity assumption is essential: the result does not extend to all 
minimal IETs. The factor\,  $n$\,  in (A1)  is optimal 
(e.\,g., it cannot be replaced by \ $n\,\ln(\ln(\ln n))$.

On the other hand, for Lebesgue almost all  $3$-IETs  $(\ui,T)$   we prove that 
for all   $\eps>0$
\[
\liminf_{n\to\infty}\limits\, n^\eps |T^n(x)-T^n(y)|=
\infty,\quad  
\text{for Lebesgue a.\,a. } (x,y)\in \ui^2. \tag{A2}
\]
This  should be contrasted with the equality
$
\liminf_{n\to\infty}\limits\,  |T^n(x)-T^n(y)|=0,
$
for a.\,a. $(x,y)\in \ui^2$,   which holds since  $(\ui^2, T\times T)$  is
ergodic   (because  generic  $3$-IETs   $(\ui,T)$  are weakly mixing).

We also prove that no  3-IET  is  strongly topologically mixing.
\end{abstract}



\tableofcontents

\section{\bf \large Introduction}
Let $(X,d)$  be a metric space with a Borel probability measure $\mu$ on it.
Let $\alp\geq0$ be a constant. Let $T\colon X\to X$ be a  Borel
measurable  map preserving measure $\mu$.
We study the distribution of the following functions (called the connectivity, 
proximality and recurrence gauges, respectively):

\begin{itemize}
\item $\phi_\alp(x,y)=\liminf_{n\to\infty}\limits\, n^{\alp} d(T^{n}(x),y), 
      \hsp{13.3} x,y\in X$;
\item $\psi_\alp(x,y)=\liminf_{n\to\infty}\limits\, n^{\alp} d(T^{n}(x),T^{n}(y)), 
      \hsp6 x,y\in X$;
\item $\rho_\alp(x)=\phi(x,x)=\liminf_{n\to\infty}\limits\, n^{\alp} d(T^{n}(x),x), 
      \hsp2 x\in X$.
\end{itemize}

We show that for $\alp>0$ the values taken by the first two gauges $\mu\ttimes\mu$-almost 
surely lie in the set $\{0,\infty\}$ (Theorem~\ref{thm:gsoinf}, Corollary~\ref{thm:gsoinf}); 
under some mild conditions all three gauges are almost surely constant 
(Section~\ref{sec:general}). 

If $X$ is a smooth manifold with a probability Borel measure $\mu$ on it and if 
$T: X\to X$ is Borel measurable, $\mu$-ergodic map which is differentiable (not 
necessarily continuously) $\mu$-almost everywhere then for any $\alp>0$ either 
$\phi_{\alp}(x,y)=0$ a.\,s., or\, $\phi_{\alp}(x,y)=\infty$ a.\,s..  (See 
Corollary~\ref{thm:connpower} in Section~\ref{sec:general} for a more general result; 
here ``a.\,s.'' stands for ``for $\mu\ttimes\mu$-a.\,a. pairs $(x,y)\in X^2$\,'').

The above general results are applied to IETs (interval exchange transformations) and more accurate results are obtained. In particular, if $T$ is an IET and $\mu$ is a 
$T$-ergodic measure then the result quoted in the previous paragraph implies only that 
each of the gauges $\phi_\alp(x,y)$ is a.\,s. a constant, either $0$ or $\infty$. 
It turns out (Theorem \ref{thm:phi1mu}) that for $\alp=1$ (in the setting of 
$\mu$-ergodic IETs) the constant is always $0$:
\[
   \phi_1(x,y)=\liminf_{n\to\infty}\limits\, n\,d(T^{n}(x),y)=
   \liminf_{n\to\infty}\limits\, n|T^n(x)-y|=0 \pmod{\mu\ttimes\mu}.
\]

The conditions of the above statement are proper. The measure $\mu$ cannot be replaced by 
Lebesgue measure, even for minimal IETs (Theorem \ref{thm:chaikaex}). Moreover, the 
``scale sequence'' \, $n$\, in the the statement of this result is optimal 
(Theorem~\ref{thm:noptim}); in particular, it cannot be replaced  by $n\,\log(\log n)$.
Related issues have been considered in \cite{D.Kim}, \cite{log iet}, \cite{AU} and 
\cite{kurz iet}. This is discussed in Section~\ref{sec:resiets}.

We also prove that for any $\alp>0$ the relation\, 
$\psi_{\alp}(x,y)=\infty$  holds 
for Lebesgue almost all $3$-IETs and for Lebesgue almost all pairs $(x,y)$ 
(Theorem \ref{thm:3psia}). One should mention that almost every 3-IET is weakly mixing 
\cite{KaS} and therefore\, 
$\psi_{0}(x,y)=0$\, holds for almost all pairs $(x,y)$. This is also discussed in 
Section~\ref{sec:resiets}.

Lastly, we show that no 3-IET is topologically mixing (Theorem \ref{thm:3nomix}).
(Note that the second author J.~Chaika has recently constructed a topologically 
mixing $4$-IET \cite{chaika3}).

Our motivation in studying the gauges $\phi$ and $\psi$  partially came from 
\cite{bosh3}  where the connection between the Hausdorff dimension of the invariant 
measure $\mu$ and the distribution of the recurrence gauge $\rho$ was established 
(see Proposition \ref{Hdim} for a related result).

In this paper we establish a variety of general results for the connectivity and 
proximality gauges (Sections \ref{sec:extreme} and especially \ref{sec:general}) in 
addition to strong results for IETs that are of independent interest and showcase the 
applications of these methods 
(Sections \ref{sec:resiets}, \ref{sec:prox} and \ref{sec:tau}).

\section{\bf \large Notation and Definitions}\label{sec:notation}                        \mbc{sec:notation}
Denote by  $\R$, $\Z$\,  the sets of real and integer numbers, respectively. 
Denote by $\ui$  the unit interval (circle) $\ui=[0,1)=\R\!\setminus\!\Z$
and by $\lam$ the Lebesgue measure on it.
\subsection{Systems and Invariant Measures}                                            
In what follows, $X,Y$ are separable metrizable spaces.                                   \mbc{1.1\\systems}

$\BB(X)$ stands for the  family ($\sigma$-algebra) of Borel subsets of $X$. 

$\MM(X)$ stands for the set of Borel probability      
        $\sigma$-additive measures on  $X$  (defined on $\BB(X)$). 
        
For a measure  $\mu\in\MM(X)$, we write $\supp(\mu)\subset X$  for the (minimal closed) 
support of $\mu$. If  $\supp(\mu)$  is a singleton $\{x\}\subset X$, the measure $\mu$  
is called atomic and is denoted by $\delta_x$. 

A map  $f\colon X\to Y$  is called measurable if it is  Borel measurable, i.\,e.\ 
if\, 
$
A\in\BB(Y)\Rightarrow f^{-1}(A)\in \BB(X).
$
Given such a  map, denote by  $f^*\!\colon\MM(X)\to\MM(Y)$ the (push forward) map          \mbc{eqs:fall}
defined by the formula
\begin{subequations}\label{eqs:fall}
\begin{equation}\label{eq:fstar}
(f^*(\mu))(B)=\mu(f^{-1}(B)) \hsp2 \text{ (for all }\ \mu\in\MM(X),\, B\in\BB(Y)),
\end{equation}
and denote by $\hat f(\mu)$  the support of the measure  $f^{*}(\mu)$:                     \mbc{eq:fstar\\eq:fhat}
\begin{equation}\label{eq:fhat}
\hat f(\mu)=\supp(f^*(\mu))\subset Y \hsp2 (\text{for }\mu\in\MM(X)).        
\end{equation}


\end{subequations}

If  $f^{*}(\mu)=\delta_y$ for some $y\in Y$ (i.\,e. if $f^{*}(\mu)$ is an atomic measure),
$f$ is called $\mu$-constant (or just {\em a.\,s.\ constant}\, if there is no doubt as to 
the measure implied). 

For two measurable maps  $f,g\colon X\to Y$ we write  $f=g\pmod{\mu}$ to signify 
the relation $f(x)=g(x)$  for $\mu$-a.\,a.\ $x\in X$.

By a metric system  $(X,T)=(X,d,T)$  we mean a metric space $X=(X,d)$ together 
with a measurable map\,  $T\colon X\to X$. Denote by 
\[
\MM(X,T)=\MM(T)=\{\mu\in\MM(X)\mid T^*(\mu)=\mu\}
\]
the set of $T$-invariant Borel probability measures on $X$.                              \mbc{1.2}
(Note that $\MM(X,T)=\emptyset$ is possible.)                                

The following abbreviations are used: \vspm1
\begin{itemize}
\item {\bf \mms} ({\em metric measure space} $(X,\mu)$) -- a metric space $X$
           (together) with a measure  $\mu\in\MM(X)$,
\item {\bf \mmps} ({\em metric measure preserving system} $(X,T,\mu)$) --
           a metric system  $(X,T)$ with an invariant measure  $\mu\in\MM(X,T)$.
\end{itemize} 

Given a metric space $(X,T)$, a measure $\mu\in\MM(T)$ is called $T$-ergodic if 
$\mu(S)\in\{0,1\}$ for every \mbox{$T$-invariant} measurable set $S\in\BB(X)$. 
An equivalent condition is that there is no presentation  $\mu=\frac{\mu_{1}+\mu_{2}}2$
with \mbox{$\mu_{1},\mu_{2}\in\MM(X)$}, $\mu_{1}\neq\mu_{2}$.

\subsection{Scale Sequences and Gauges.}\label{subs:scale}                                            
\begin{defi} \label{def:scale} 
By a {\em scale sequence} we mean a sequence  $\bfs=\{s_n\}_1^\infty$ of positive             \mbc{subs:scale}
real numbers such that \mbox{$\lim_{n\to\infty}\limits s_n=\infty$}.                          \mbc{def:scale} 
For $\alp>0$, the sequence $\bfs_\alp=\{n^\alp\}_1^\infty$ will be referred to as
the $\alp$-power sequence.                                                                    \mbc{1.25}                       
\end{defi}

\begin{defi}\label{def:gauge}
Let  $(X,d)$ be a metric space. By a {\em gauge} (on $X$) we mean a measurable map          \mbc{def:gauge}
$f\colon X\to[0,\infty]$. Let $\mu\in\MM(X)$ be a measure.
A~gauge  $f$  is called \vspm1
{\em
\begin{description}
\item[$\boldsymbol\mu$-extreme] \hsp2 if\, $\hat f(\mu)=\supp(f^*(\mu))\subset\{0,\infty\}$ 
                     (i.\,e.\ if\, $\mu(f^{-1}(\{0,\infty\})=1$);
\item[$\boldsymbol\mu$-constant] \hsp{1.5} if\, $f^{*}(\mu)$ is an atomic measure, i.\,e.\ if\, 
                     $f^{*}(\mu)=\delta_c$ for some $c\in[0,\infty]$;
\item[$\boldsymbol\mu$-trivial] \hsp{5.7}   if\, $f^{*}(\mu)\in\{\delta_0,\delta_\infty\}$.
\end{description}
}
\end{defi}
\vspm1

When using the above terminology we suppress referring to the measure $\mu$ if there 
are no doubts as to the measure $\mu$ implied. Thus, given a measure $\mu\in\MM(X)$,
a gauge on $X$ is trivial if and only if it is both extreme and constant.

   Given a metric system  $(X,T)$  and a scale sequence  $\bfs=\{s_n\}_1^\infty$, two 
gauges $\phi,\psi$ on $X^2=X\ttimes X$  and one gauge $\rho$ on $X$ are introduced as         \mbc{eqs:gauges}
follows:                                                                
\begin{subequations}\label{eqs:gauges}                                                  
\begin{align}
\phi(x,y)=\Phi(x,y,\bfs,T)=&\,
\liminf_{n\to \infty}\limits \, s_n\,d(T^n(x),y)&
&\text{(\small the connectivity gauge),}\hsp9  \label{eq:phi}\\
\psi(x,y)=\Psi(x,y,\bfs,T)=&\,
\liminf_{n\to \infty}\limits s_n\,d(T^n(x),T^n(y))&
&\text{(\small the proximality gauge),}  \label{eq:psi}\\
\rho(x)=\Phi(x,x,\bfs,T)=&\,
\liminf_{n\to \infty}\limits s_n\,d(T^n(x),x)&
&\text{(\small the recurrence gauge).}  \label{eq:rho}
\end{align}

The above gauges are measurable (because  $T$  is),  so that, given measures              \mbc{eq:gaugestar}
$\mu,\nu\in \MM(X)$, the following measures in $\MM([0,\infty])$\vspm6
\begin{equation}\label{eq:gaugestar}
\phi^*(\mu,\nu)\bydef \phi^*(\mu\ttimes\nu),
\qquad
\psi^*(\mu,\nu)\bydef \psi^*(\mu\ttimes\nu),
\qquad
\rho^*(\mu)
\end{equation}
\end{subequations}
and the subsets $\hat\phi(\mu,\nu),\,\hat\psi(\mu,\nu),\,\hat\rho(\mu)\subset[0,\infty]$
are well  defined (see \eqref{eqs:fall}; here   $\mu\ttimes\nu\in \MM(X^2)$              \mbc{1.3}
stands for the product measure).             
     
We often add subscript  $\alp>0$  to functional symbols ($\phi$, $\psi$ or $\rho$)       \mbc{eqs:gaugesa}
of the above three gauges to specify the $\alp$-power sequence  
$\bfs=\bfs_\alp=\{n^\alp\}_1^\infty$  to be used as a scale sequence:       
\begin{subequations}\label{eqs:gaugesa}
\begin{align}
\phi_\alp(x,y)&=\Phi(x,y,\bfs_\alp,T)=
\liminf_{n\to \infty}\limits \, n^\alp\,d(T^n(x),y)&
&\text{(\small the $\alp$-connectivity gauge),}&
\label{eq:phia}\\
\psi_\alp(x,y)&=\Psi(x,y,\bfs_\alp,T)=
\liminf_{n\to \infty}\limits n^\alp\,d(T^n(x),T^n(y))&
&\text{(\small the $\alp$-proximality gauge),}&
\label{eq:psia}   \\
\rho_\alp(x)&=\Phi(x,x,\bfs_\alp,T)=
\liminf_{n\to \infty}\limits n^\alp\,d(T^n(x),x)&
&\text{(\small the $\alp$-recurrence gauge).}&
\label{eq:rhoa}   
\end{align}
\end{subequations}

Note that the equations \eqref{eqs:gaugesa} make sense and define gauges for $\alp=0$
(even though the constant sequence $s_n=1$ is not a scale sequence).

A remarkable fact is that for an arbitrary $\mmps$\ $(X,T,\mu)$ and all $\alp>0$, 
the $\alp$-connectivity and $\alp$-proximality gauges ($\phi_\alp$ and $\psi_\alp$) are always 
extreme (Theorem~\ref{thm:gsoinf}) while the same assertion does not need to hold for the  
$\alp$-recurrence gauges $\rho_\alp$ (Example \ref{ex:gm}, page \pageref{ex:gm}).

Given a \mmps\ $(X,T,\mu)$, we define  three constants 
(taking values in the set\, $[0,\infty]$) in the following way:                             \mbc{eqs:3const}
\begin{subequations}\label{eqs:3const}
\begin{align}
C_\phi&=C_\phi\big((X,T,\mu)\big)=\sup\left\{\{0\}\cup\{\alp>0\mid
       \phi^*_\alp(\mu,\mu)=\delta_0\}\right\}&
       &\text{(\small the connectivity constant)}&\\
C_\psi&=C_\psi\big((X,T,\mu)\big)=\sup\big\{\{0\}\cup\{\alp>0\mid 
       \psi^*_\alp(\mu,\mu)=\delta_0\}\big\}&
       &\text{(\small the proximality constant)}&\label{eqs:3constb}\\
C_\rho&=C_\rho\big((X,T,\mu)\big)=\sup\big\{\{0\}\cup\{\alp>0\mid 
       \rho^*_\alp(\mu)=\delta_0\}\big\}&
       &\text{(\small the recurrence constant)}&
\end{align}
\end{subequations}
\vsp1
According to notation in \eqref{eqs:fall}, \eqref{eq:gaugestar} and \eqref{eqs:gaugesa}, 
$\phi^*_\alp(\mu,\mu)=\delta_0$\, is a shortcut for 
$\phi_{\alp}(x,y)=0\!\pmod{\mu\ttimes\mu}$.  Both expressions mean that\,
$\phi_{\alp}(x,y)=\liminf_{n\to\infty}\limits\, n^\alp d(T^n(x),y)=0$,
for\, $\mu\ttimes\mu$-a.\,a.\ $x,y\in X$. The same interpretation applies for the 
expressions $\psi^*_\alp(\mu,\mu)=\delta_0$ and  $\rho^*_\alp(\mu)=\delta_0$. 

The recurrence constant has been previously connected to the Hausdorff dimension of 
a system \cite{bosh3} and to entropy for symbolic dynamical systems \cite{O-W}. 
This paper links the connectivity and proximality constants to the Hausdorff dimension 
(Proposition \ref{Hdim} and Remark \ref{rem:Hdim prox}).

It also links the proximality 
constant to complexity properties (Propositions \ref{proxiet} and \ref{proxlinrec}).


The terminology chosen in the paper (connectivity, proximality and recurrence) is          \mbc{\large manife-\\station?}
suggestive: larger gauge constants indicate stronger manifestation of the corresponding 
property. For instance the proximality constant of any distal system is 0.



\section{\bf \large The extremality of the connectivity and the proximality gauges}
\label{sec:extreme}
In this section we prove that the ``power'' gauges $\phi_\alp$ and $\psi_\alp$ are always       \mbc{sec:extreme}
extreme (Theorem \ref{thm:gsoinf}). The main tool is the EGT (the Extremality Gauge 
Theorem, Theorem~\ref{thm:egt}), a result of independent interest. The proof of the EGT 
is based on the method of ``a single orbit analysis'' in the spirit of  the book \cite{WBe}.
\subsection{Properties of Scale Sequences} 
Many results regarding the gauges $\phi$, $\psi$ and $\rho$ will be proved in a larger 
generality than scales sequences $\bfs=\bfs_\alp=\{n^\alp\}$. We need the following 
definitions detailing some properties a scale sequence may satisfy. 
\vsp{1}

\begin{defi} \label{def:scalep}                 
A scale sequence $\bfs=\{s_n\}_1^\infty$ is called:                                         \mbc{def:scalep}
\vsp{.5}
\begin{description}
\item[monotone] \hsp8 if \ $s_{n+1}\geq s_n$, for\, $n\in\N$\, large enough,\vsp1
\item[steady] \hsp{14}   if\, $\lim_{n\to\infty}\limits\frac{s_{n+1}}{s_{n}}=1$,
\item [two-jumpy] \hsp{6.6}  if \ $\bfs$ \ is monotone and if \
    $\liminf_{n\to\infty}\limits\frac{s_{2n}}{s_{n}}>1$, 
\item [bounded-ratio] \hsp1 if \ $\sup_{n\geq1}\limits \frac{s_{n+1}}{s_{n}}<\infty$, 
\item [nice] \hsp{19} if  \ $\bfs$ \ is two-jumpy and bounded-ratio. \end{description}
\end{defi}
\vsp1

Power scale sequences $\bfs_\alp$  are ``all of the above'': 
monotone, steady, two-jumpy, bounded-ratio and nice. 
The sequence  $\{2^n\}$ is nice but not steady. The sequence of primes $\{2,3,5,\ldots\}$
is ``all of the above''.

\begin{thm}\label{thm:gsoinf}
Let  $(X, T)$  be a metric system and let\, $\bfs=\{s_n\}$ be a two-jumpy scale sequence.   \mbc{thm:gsoinf}
Let both \mbox{$\mu,\nu\in\MM(X,T)$} be $T$-invatiant measures. Then the connectivity 
and the proximality gauges ($\phi$ and $\psi$) are both \mbox{$\mu\ttimes\nu$-extreme}. In 
particular, for all\, $\alp>0$, both $\phi_\alp$ and $\psi_\alp$ are $\mu\ttimes\nu$-extreme.
\end{thm}
\begin{cor}\label{cor:gsoinf}
Let  $(X, T, \mu)$  be a \mmps\ and let $\bfs=\{s_n\}$ be a two-jumpy scale sequence.       \mbc{cor:gsoinf}
Then both gauges $\phi,\psi$  are extreme  (i.e., $\mu\ttimes\mu$-extreme). 
In particular, for all\, $\alp>0$, both $\phi_\alp,\psi_\alp$ are extreme.
\end{cor}
As the following example shows, the analogue of the Theorem \ref{thm:gsoinf} fails          \mbc{1.46}
for the recurrence gauge. In what follows, we write  $\fracpart{\cdot}$ for the 
distance to the closest integer: 
$\fracpart{x}=\min_{n\in\Z}\limits |x-n|$,\, for $x\in\R$.\vspm{1.5}
\begin{ex}\label{ex:gm}
Let $(\ui,T)$ be the golden mean rotation $T(x)=x+\alp\pmod1,\,\alp=\frac{\sqrt5-1}2$,      \mbc{ex:gm}
of the unit interval (circle)\, $\ui=[0,1)$. Then 
$\rho_1(x)=\liminf_{n\to\infty}\limits\, n\fracpart{n\alp}=\frac1{\sqrt5}$, for all 
$x\in\ui$; see e.\,g. \cite[Chap.1, \S6]{Ca}. Thus the \mbox{$1$-recurrence} 
gauge for the \mmps\ $(\ui,T,\lam)$ is constant but not extreme. (Note that 
the Lebesgue measure  $\lam\in\MM(\ui,T)$  is $T$-invariant). 
\end{ex}
Theorem \ref{thm:gsoinf}  follows easily from the EGT (Theorem \ref{thm:egt} below).       \mbc{1.48}
\vsp1 

\subsection{The Extremality Gauge Theorem (EGT)}
One needs the following definition (cf.\ Definition~\ref{def:gauge}).                  
\begin{defi}\label{def:gauge2}                                      
Let  $(X,\BB,\mu)$  be a probability measure space. A map  $f\colon X\to[0,\infty]$         \mbc{def:gauge2}
is called a gauge (on $X$) if  $f$ is measurable (i.\,e., for every open subset  
$U\subset \R$, $f^{-1}(U)\in\BB$). 
Such an $f$  is said to be an extreme gauge if\, $\mu(f^{-1}(\{0,\infty\})=1$.
\end{defi}
  The following theorem (the Extremality Gauge Theorem) claims that certain procedure
to derive new gauges always leads to extreme gauges; this theorem plays an important 
role in this paper.
\begin{thm}[EGT] \label{thm:egt}
Let $(X,\BB,\mu,T)$ be a measure preserving system and let $f\colon X\to[0,\infty]$         \mbc{thm:egt}
be a gauge. Let\, $\bfs=\{s_n\}_1^\infty$ be a two-jumpy scale sequence. Then             
the gauge \ $F(x)=\liminf_{n\to\infty}\limits\, s_n{f(T^nx)}$ \ is extreme.
\end{thm}
\begin{rem}\label{rem:gm}
The assumption for $\bfs$ to be two-jumpy is important. The claim of                        \mbc{rem:gm}
Theorem~\ref{thm:egt} fails for the scale sequence\, $\bfs=(\log(n+1))_1^{\infty}$ 
and the golden mean rotation  $(\ui,T,\lam)$ (see Example \ref{ex:gm} above) with            \mbc{1.5}
$f(x)=-1/\log\fracpart{x}$. (It is not hard to show that in this case     
$F(x)\equiv1 \pmod{\lam}$). In fact there are similar examples for almost every 
IET \cite{log iet}. These `logarithm laws' are widely studied and one can view 
some of this paper's results as exploring finer analogues.

On the other hand, the EGT (Theorem \ref{thm:egt}) holds if  $\bfs=\{s_n\}_1^\infty$ is        \mbc{1.51}
a power scale sequence:  $s_n=n^\alp$, with some $\alp>0$. More generally, if a sequence       
$\bfs=\{s_n\}$  is regularly growing  (in the sense e.\,g. that it is defined as 
the restriction  $s_n=g(n)$ of a function $g(x)$  in a Hardy field,  see e.\,g. 
\cite{hardy unif} for introduction in the subject), then a necessary and sufficient condition 
for a scale sequence $\bfs$ to be two-jumpy is that $s_n>n^\alp$, for some $\alp>0$  
and all large~$n$. 
\end{rem}\vspm1
\begin{proof}[\bf Proof of Theorem \ref{thm:gsoinf}]
Set $Y=X^2$ and consider the maps $T'\colon Y\to Y$ and $T''\colon Y\to Y$ defined         \mbc{proof\\thm:gsoinf}
by the formulae $T'(x_1,x_2)=(Tx_1,x_2)$ and  $T''(x_1,x_2)=(Tx_1,Tx_2)$. Clearly, 
$\mu\ttimes\mu\in\MM(Y,T')$ and  also  $\mu\ttimes\mu\in\MM(Y,T'')$. 
In order to prove the inclusion $\hat\phi(\mu,\mu)\subset\{0,\infty\}$, one applies
Theorem~\ref{thm:egt} with $T=T'$, $X=Y$ and $f(\bfx)=d(x_1,x_2)$ where\,
$\bfx=(x_1,x_2)\in Y$.
To prove the inclusion $\hat\psi(\mu,\mu)\subset\{0,\infty\}$, one applies
Theorem~\ref{thm:egt}  with  $T=T''$ and with the same  $X$ and $f$\, 
($X=Y$ and $f(\bfx)=d(x_1,x_2)$).
\end{proof}
\begin{proof}[\bf Proof of Theorem {\ref{thm:egt}}]
Under the standard convention that $\frac10=\uvm{.8}\infty$ and 
$\frac1{\uvm{.77}\infty}=0$, set the gauges  $g(x)=1/f(x)$ and
$
G(x)=\limsup_{n\to\infty}\limits\, \dfrac{g(T^nx)}{s_n}.
$ 
Since\,  $G(x)=1/F(x)$,  it would suffice to prove that the gauge $G(x)$  is extreme.        \mbc{1.54}

Without loss of generality we may assume that the system $(X,T)$ is ergodic
(by passing to the ergodic decomposition of $(X,T)$). 

Consider the sets\, 
$S=g^{-1}(\infty)=\{x\in X\mid g(x)=\infty\}$, 
$S_1=\bigcup_{n=0}^\infty\limits T^{-n}(S)$\, and\, 
$S_2=\bigcap_{n=0}^\infty\limits T^{-n}(S_1)$.

If $\mu(S)>0$, then $\mu(S_1)=\mu(S_2)=1$ in view of the ergodicity of $T$.
Since $G(x)=\infty$ for all $x\in S_2$, it follows that  $G^*(\mu)=\delta_\infty$,
i.\,e. that $G$  is extreme.
It remains to consider the case $\mu(S_1)=0$. 

By replacing $X$ with $X\setminus S_1$,
we may assume without loss of generality that $S=\emptyset$. For $k\geq1$, denote \vspm3
\[
G_k(x)=\sup_{n\geq1}\limits\, \frac{g_k(T^nx)}{s_n},
\hsp8 \text{where }\
g_k(x)=
\begin{cases}
g(x),&\text{if }\, g(x)\geq k\\
0,&\text{otherwise.}
\end{cases}
\]

Observe that  $\big\{G_k(x)\big\}_1^\infty$  is a non-increasing sequence of gauges on $X$           \mbc{1.55}
pointwise converging to~$G(x)$.  On the other hand, $G(T(x))\geq G(x)$ (because 
$\bfs$ is increasing). By the ergodicity of $T$, $G(x)$ must be constant:
$G(x)= c\!\pmod{\mu}$, for some\,  $c\in[0,\infty]$.
We have to show that $c\in\{0,\infty\}$. 

Assume to the contrary that\, $0<c<\infty$.
Since $\bfs$ is two-jumpy, we have $\liminf_{n\to\infty}\limits \frac{s_{2n}}{s_n}>1$.
Observe that the definition of $G$ is insensitive to modifications
of a finite number of the terms of sequence\, $\bfs$. Therefore
we may assume that there exists a constant $M>1$  such that $\frac{s_{2n}}{s_n}>M$ 
for all $n\geq1$. Set 
\begin{align*}
    \tfrac12>\eps&=\,\tfrac{M-1}{2M}>0, \hsp{1.5}  c'=M(1-\eps)\,c=\tfrac{(M+1)}2\,c>c,\\[1mm]
    A_k&=\,\{x\in X\mid \, G_k(x)\geq c\},\hsp4  A=
        \uvm1{\bigcap_{k\geq1}\nolimits} A_k,\\[1mm]
    B_k&=\,\{x\in X\mid \, G_k(x)\geq c'\}, \hsp{2.7} B=
        \uvm1{\bigcap_{k\geq1}\nolimits} B_k.
\end{align*}

Since $\big\{G_k(x)\big\}_1^\infty$ is a non-increasing sequence of functions converging to the     \mbc{1.57}
constant\, $c$\, and since $c'>c$, it follows that the sequences of sets 
$(A_k)_1^\infty$ and $(B_k)_1^\infty$ are non-increasing: $A_k\supset A_{k+1}\supset 
A,\ B_k\supset B_{k+1}\supset B$ ($k\geq1$), and that $\mu(A)=1$ and $\mu(B)=0$.

Now let $x\in A$.  Fot every $k\geq 1$, we have  $x\in A_k$ 
and hence one can select $n_k\geq1$  such that  
\[
\frac{g_k(T^{n_k}x)}{s_{n_k}}\geq c(1-\eps).
\]
Since  $c(1-\eps)>0$, it follows that  $g(T^{n_k}x)\geq k$ and therefore
$\lim_{k\to\infty}\limits n_k=\infty$. 
(Here we use the assumption that $S=g^{-1}(\infty)=\emptyset$). Set
\[
    I_k=\big[\tfrac{n_k}2,n_k\big)\cap\N=\left\{m\in \N\mid 
    \tfrac{n_k}2\leq m<n_k\right\}, \qquad
    J_N=\bigcup_{k\geq N} I_k, \hsp2  N\geq1.
\]
\vspm4
We claim that  $T^m(x)\in B_k$   for $m\in I_k$.  Indeed,                                   \mbc{1.59}
\[
    \frac{g_k\big(T^{n_k-m}(T^mx)\big)}{s_{n_k-m}}=\frac{g_k(T^{n_k})}{s_{n_k}}\cdot
    \frac{s_{n_k}}{s_{n_k-m}}\geq c(1-\eps)\cdot\frac{s_{2(n_k-m)}}{s_{n_k-m}}
    \geq c(1-\eps)M=c'.
\]

Fix $N\geq1$. It follows that $T^m(x)\in B_N$ for all $m\in J_N=\bigcup_{k\geq N} I_k$ 
(since $\{B_k\}$ is a non-increasing sequence of sets). Note that $J_N\subset\N$  is 
a subset of upper density at least $1/2$.                                               
Since $x\in A$  is arbitrary and $\mu(A)=1$,  the ergodicity of $T$  implies
$\mu(B_N)\geq1/2$  which is in contradiction with  
\mbox{$\lim_{N\to\infty}\limits \mu(B_N)=\mu(B)=0$}.
\end{proof}\vsp0

\subsection{Proximality Constant of a Weakly Mixing \mmps}
\begin{prop}[The invariance of the the proximality gauge]\label{prop:disinv}
Let $(X,T,\mu)$  be a \mmps\ and let \mbox{$\bfs=\{s_n\}_1^\infty$}  be either a monotone or       \mbc{prop:disinv}
steady scale sequence.  Then the proximality gauge\,  
\begin{equation}\label{eq:dist3}
\psi(x,y)=\liminf_{n\to\infty}\limits\,s_n\,d(T^nx,T^ny), \quad x,y\in X,
\end{equation}  
is\, $T\ttimes T$-invariant: \  $\psi(Tx,Ty)\equiv \psi(x,y) \pmod{\mu\ttimes\mu}$.         \mbc{1.6}
\end{prop}\vsp0
\begin{prop}[The constancy of the the proximality gauge]\label{prop:dismix}
Let $(X,T,\mu)$  be a weakly mixing  \mmps\ and let $\bfs=\{s_n\}_1^\infty$  be either     \mbc{prop:dismix}
a monotone or steady scale sequence.  Then the measure $\psi^*(\mu,\mu)$ is  atomic,
i.\,e., the proximality gauge $\psi$ (see \eqref{eq:dist3}) is $\mu\ttimes\mu$-constant.
In particular, $\psi^*_\alp(\mu,\mu)=\delta_0$\,  if\, $0\leq\alp<C_\psi$ and 
$\psi^*_\alp(\mu,\mu)=\delta_\infty$  if\, $C_\psi<\alp<\infty$  where  
$C_\psi=C_\psi(X,T,\mu)$ is the proximality constant (see \eqref{eqs:3constb}).
\end{prop} 
\begin{cor}\label{cor:dismix}
Under the conditions of the above theorem, assume that the scale sequence $\bfs$ is         \mbc{cor:dismix}
two-jumpy. Then the proximality gauge is trivial, i.\,e.\ 
$\psi^*(\mu,\mu)\in\{\delta_0,\delta_\infty\}$.
\end{cor} 
\begin{proof}[\bf Proof of Proposition \ref{prop:disinv}]
Denote $Y=X^2$ and $T''=T\ttimes T$, $\nu=\mu\ttimes\mu$.  The assumption on\, $\bfs$\,      \mbc{1.61}
implies the inequality $\psi(T''y)\leq\psi(y)$, for all $y\in Y$. 
Since  $\nu\in\MM(Y,T'')$, the claim of Proposition \ref{prop:disinv} follows.
\end{proof}
\begin{proof}[\bf Proof of Proposition \ref{prop:dismix}]
Follows from Proposition \ref{prop:disinv} because  $T\ttimes T$  is ergodic.                \mbc{1.62}
\end{proof}
\begin{proof}[\bf Proof of Corollary \ref{cor:dismix}]
It is enough to show that $\psi$ is both\, $\mu\ttimes\mu$-extreme and 
$\mu\ttimes\mu$-constant. The first assertion follows from 
Theorem \ref{thm:gsoinf};  the second assertion is a consequence of 
Proposition \ref{prop:dismix}.
\end{proof}\vsp0

\section{\bf \large Results on IETs}\label{sec:resiets}                                   \mbc{sec:resiets}

\subsection{A brief introduction}
Several results in the paper concern the systems on the unit interval $X=\ui=[0,1)$, 
and in particular the systems called {\em interval exchange transformations} (IETs).                 
Let  $\{r, L, \pi\}$  be a triple such that
\begin{itemize}
\item  $r\geq2$ is an integer,
\item  $\pi\in S^r$  is a permutation on the set  $J_r=\{1,2,\ldots,r\}$,
\item  $L$ is a positive probability vector, 
     $L=(\ell_1, \ell_2,\ldots,\ell_r)\in\R^r$ with 
     $\sum_{i=1}^r \ell_i=1$ and all $\ell_i>0$.
\end{itemize}
An IET is a map $T\colon\ui\to\ui$ completely determined by the parameters $\{r,L,\pi\}$   \mbc{1.42}
in the following way. Set  $s_0=0$ and  $s_k=\sum_{i=1}^k \ell_i$, for $1\leq k\leq r$. 
This way the interval  $\ui=[0,1)$ is partitioned into $r$ subintervals 
$\ui_k=[s_{k-1},s_k)$, $1\leq k\leq r$. Define  $T=T(L,\pi)$  by the formula
\begin{equation}\label{eq:iets}
T(x)=x-\sum_{i<k} \ell_i\,+\!\!\sum_{\pi(i')<\pi(k)}\!\! \ell_{i'},
\quad \text{for }  x\in\ui_k.
\end{equation}
A map  $T$  constructed  this way is called the  $(L,\pi)$-IET (sometimes, less 
informatively,  an  $r$-IET or just an IET).
It exchanges  the intervals  $\ui_k$  in accordance with permutation  $\pi$.
It follows from the fact that IETs are invertible piecewise isomertries that $\lam\in\MM(T)$ for every IET  $T$.
(Every IET preserves the Lebesgue measure). If $\MM(T)=\{\lam\}$, a singleton,
$T$ is called uniquely ergodic. (The equivalent condition is that all orbits of $T$
are uniformly distributed). An IET $T$ is called ergodic if it is $\lam$-ergodic.
  

A permutation        $\pi\in S^r$  is called irreducible if for every                      \mbc{1.43}
$k\in J_{r-1}=\{1,2,\ldots,r-1\}$  there exists  $i\in J_{r-1}$  such that  
$i\leq k<\pi(i)$.  An IET   $T$  is called minimal if every orbit is dense.  
For  $n\geq2$, the irreducibility of  $\pi$  is a necessary condition for the 
minimality  of  $T=(L,\pi)$. On the other hand, for  irreducible  $\pi\in S_r$,  
linear independence of the  $r$  entries of  $L$  (over $\Q$)  is a sufficient 
condition for minimality  of  $(L,\pi)$,  see  \cite{IET}.  Thus  $(L,\pi)$  is 
minimal for irreducible  $\pi$  and Lebesgue almost all  $L$.

Every $\lam$-ergodic IET must be minimal but the converse does not hold. A minimal IET        \mbc{1.44}
does not need to be uniquely ergodic or even $\lam$-ergodic \cite{KN} and \cite{nonue}. Nevertheless, ``most'' 
minimal IETs are uniquely ergodic. More precisely, for an irreducible permutation $\pi$ 
the IET  $(L,\pi)$ is uniquely ergodic for Lebesgue almost all  $L$. This result 
(often referred to as Keane's conjecture) has been proved independently  by  
Mazur \cite{M}  and Veech \cite{gauss}; see also \cite{bosh1} for an elementary proof.  

For $n\geq3$ and irreducible $\pi\in S_r$ which are not rotations\, (the technical  
condition is $\pi(k)+1\neq\pi(k+1)$ for $1\leq k<r$)\, the  $(L,\pi)$-IETs  has been         \mbc{CHECK!}
recently shown to be weakly mixing for Lebesgue a.\,a.\ $L$ (see \cite{AF}). In 
particular, Lebesgue a.\,a.\ $3$-IETs with $\pi(k)=4-k$ are weakly mixing \cite{KaS}.


\subsection{The Connectivity Gauges}
Given an IET  $(\ui,T)$,  recall that\,
$\phi_1(x,y)=\liminf_{n\to \infty}\limits \, n\,|T^n(x)-y|$, for $x,y\in\ui$.   
The following is the central result on the connectivity gauges for IETs.
\begin{thm}\label{thm:phi1mu}
Let $T$  be an IET and assume that $\mu\in\MM(T)$ is a $T$-ergodic measure.                  \mbc{thm:phi1mu}
Then $\phi^*_1(\mu,\mu)=\delta_0$.
\end{thm}
Recall that $\phi^*_1(\mu,\mu)=\delta_0$  in our notation means that  
$\phi_1(T^nx,y)=0 \!\pmod{\mu\ttimes\mu}$ (Section~\ref{sec:notation}).
\begin{cor}[The case $\mu=\lam$, the Lebesgue measure]\label{cor:phi1lam}
Let $(\ui,T)$ be a $\lam$-ergodic IET. Then $\phi^*_1(\lam,\lam)=\delta_0$.                            \mbc{cor:phi1lam}                       
\end{cor}
\begin{rem}\label{rem:tseng}
A special case of Corollary \ref{cor:phi1lam} (for the irrational  rotations $T$  of the unit circle        \mbc{1.67\\rem:tseng}
$\ui$) is already known \cite{D.Kim}.  
(Irrational rotations are naturally identified with  $2$-IETs).                              \mbc{\Small CHANGE! }
\end{rem}
Note that the ergodicity assumption in Theorem \ref{thm:phi1mu} (and 
Corollary \ref{cor:phi1lam})  is crucial (see Theorem \ref{thm:chaikaex}),       
and the choice of the factor  $s_n=n$ in the formula for  $\phi(x,y)$  is optimal            \mbc{1.68\\p7}
(see Theorem  \ref{thm:noptim}).

The next theorem shows that even the special case of Corollary \ref{cor:phi1lam} does not 
extend to all minimal IETs. Nevertheless, for minimal IETs a weaker assertion 
is possible (see Proposition \ref{thm:phi1m}). 
\begin{thm}\label{thm:chaikaex}
There exists a minimal $4$-IET\, $(\ui,T)$  such that\,
$\phi^*_1(\lam,\lam)=b\delta_{\uvm{.6}0}+(1-b)\delta_{\uvm{.6}\infty}$,                     \mbc{thm:chaikaex\\5}
with some\,  $b\in\left[\frac14,\frac34\right]$.
\end{thm} 
Clearly, in the above theorem $b=(\lam\ttimes\lam)(\phi_1^{-1}(0))$.\vsp1
\begin{rem}\label{rem:twodelta}
Note that for every \mmps\ $(X,T,\mu)$ and for any $\alp>0$,  the measures  
$\phi^*_\alp(\mu,\mu)$, $\psi^*_\alp(\mu,\mu)$ have the form  $b\del_0+(1-b)\del_\infty$, 
$0\leq b\leq1$, because the gauges $\phi_\alp$, $\psi_\alp$ are extreme 
(see Corollary \ref{cor:gsoinf}, page~\pageref{cor:gsoinf}).\vsp1
\end{rem}

In what follows,  $[r]$ stands for the integer  part of\, $r\in\R$:\,
$[r]=\min\{k\in\Z\mid k\geq r\}$.                                                         \mbc{thm:phi1m\\6}
\begin{prop}\label{thm:phi1m}      
Let\, $(\ui,T)$ be a minimal  $r$-IET. Then \                                             \mbc{1.7}
$  (\lam\ttimes\lam)(\phi_1^{-1}(0))=\phi^*(\lam,\lam)\left(\{0\}\right)\geq 
   \uvm{1.15}{\frac1{[r/2]}}\geq 2/r >0$.                          
\end{prop}
Denote by  $R_\beta\colon\ui\to\ui$ the $\beta$-rotation map 
$R_\beta(x)=x+\beta\!\pmod 1$. (It can be viewed as a $2$-IET).
\begin{thm}\label{thm:noptim}                                                       
Let  $\bfs=\{s_n\}_1^\infty$  be a scale sequence such that\,                              \mbc{thm:noptim\\7}
$\lim_{n\to\infty}\limits\frac{s_n}n=\infty$. Then there exists an irrational 
$\beta\in\R\!\setminus\!\Q$\,  such that the $\beta$-rotation $T=R_\beta$ satisfies 
$\phi^{*}(\lam,\lam)=\delta_\infty$.
\end{thm}
Recall that $\phi^*(\lam,\lam)=\delta_{\uvm{.6}\infty}$ means that  
$\phi(x,y)=\liminf_{n\to \infty}\limits\, s_n\,d(T^n x,y))=\infty \pmod{\lam\toto}$.
  
Note that, under the conditions of Theorem \ref{thm:noptim}, the relation
$\phi^{*}_1(\lam,\lam)=\delta_0$ holds (by Corollary \ref{cor:phi1lam})  even though  
\mbox{$\phi^{*}(\lam,\lam)=\delta_\infty$}. 
It follows that in Proposition \ref{thm:phi1mu} and Corollary \ref{cor:phi1lam} 
the factor $n$  (underlined in the formula\, 
$\liminf_{n\to\infty}\limits\, \underline n\,|T^nx-y|$\, for~$\phi_1(x,y)$)
cannot be replaced by one  approaching infinity faster.\vsp2

\begin{rem}\label{rem::noptim}
The conclusion of Theorem \ref{thm:noptim} can be presented in the following equivalent   \mbc{1.73}
form: There exists an irrational $\alp\in\ui=[0,1)$  such that
$\lim_{n\to\infty}\limits s_n\fracpart{n\alp-y}=\infty$  for $\lam$-a.\,a. $y\in\ui$.
\end{rem}

To put Proposition \ref{thm:phi1mu} and Theorems \ref{thm:chaikaex} and \ref{thm:phi1m}
in perspective, we state the following three general propositions. 
\begin{prop}\label{thm:genlam} Let  $\bfs=\{s_n\}_1^\infty$  be a scale sequence,          \mbc{thm:genlam\\8}
let $(\ui,T)$ be a metric system (on the unit interval). \mbox{Let~$\mu\in\MM(\ui)$.} 
Then the gauge\, $\phi(x,y)=\liminf_{n\to\infty}\limits\, s_n|d(T^nx,y)|$ is extreme 
relative to the measure\,  \mbox{$\mu\times\lam\in\MM(\ui^2)$}  (i.\,e., 
$\hat\phi(\mu,\lam)\subset\{0,\infty\}$). 
\end{prop}\vsp{.5}
Note that in Proposition \ref{thm:genlam} no assumptions are imposed on the scale sequence
$\bfs$,  and the measures $\mu$, $\lam$ are not assumed to be $T$-invariant (cf.\ EGT, 
Theorem \ref{thm:egt}). Proposition \ref{thm:genlam} is a special case of 
Theorem \ref{thm:gendec} (presented and proved in Section \ref{sec:dec}).                 \mbc{1.75}
\begin{prop}\label{thm:geniets2}                                                         
Let  $\bfs=\{s_n\}_1^\infty$  be a steady scale sequence, let  $(\ui,T)$ be an $r$-IET    \mbc{thm:geniets2}
and let $\mu,\nu\in\MM(T)$ (be invariant measures). Then the set\,
$\hat\phi(\mu,\nu)\subset[0,\infty]$  must be finite.
\end{prop}\vspm1
This follows from Proposition \ref{thm:ergiets} and the fact that there are at most $\frac r 2$ non-atomic ergodic measures.

Note that if $\bfs$ is two-jumpy then the cardinality of $\hat\phi(\mu,\nu)$ does not 
exceed $2$ because of the inclusion $\hat\phi(\mu,\nu)\subset\{0,\infty\}$ holds          
by Theorem \ref{thm:gsoinf}.
\begin{prop}\label{thm:ergiets}                                                         
Let  $\bfs=\{s_n\}_1^\infty$  be a steady scale sequence, let  $(\ui,T)$ be an IET and    \mbc{thm:ergiets\\10}
let $\mu,\nu\in\MM(T)$ be two ergodic measures. Then the measure $\phi^{*}(\mu,\nu)$
is atomic, i.\,e.\ the set\, $\hat\phi(\mu,\nu)$  must be a singleton in $[0,\infty]$.    \mbc{1.8}
\end{prop}
Note that the claim of the above theorem holds in a more general setting of 
\mmps\ which are local contractions (see Theorem \ref{thm:conn1} in Section 
\ref{sec:general}). The next proposition shows that under certain assumption 
on  $\bfs$ the conclusion of Proposition \ref{thm:ergiets} can be strengthened.
\begin{prop}\label{thm:ergiets2}                                                         
Let  $\bfs=\{s_n\}_1^\infty$  be a nice scale sequence, let  $(\ui,T)$ be an IET and       \mbc{thm:ergiets\\11}
let $\mu,\nu\in\MM(T)$ be two ergodic measures. Then 
$\phi^{*}(\mu,\nu)\in\{\delta_0, \delta_\infty\}$.    										   \mbc{1.8}
\end{prop}
Recall that $\bfs$ is called nice if it is bounded-ratio and two-jumpy 
(Definition \ref{def:scale}). Under the stronger condition that $\bfs$ is steady 
and two-jumpy the claim of Proposition \ref{thm:ergiets2} follows immediately from 
Theorem \ref{thm:gsoinf} and Proposition \ref{thm:ergiets}. 
In Section \ref{sec:general} a more general version of Proposition \ref{thm:ergiets2}
is stated and proved (Theorem \ref{thm:conn2}).

By Theorem \ref{thm:phi1mu}, $\phi_1^{*}(\mu,\mu)=\delta_0$, for any ergodic $\mu\in\MM(T)$.
It is possible to have  $\phi^{*}_1(\mu,\nu)=\delta_\infty$  for some minimal IET $T$ and 
$T$-ergodic measures\, $\mu,\nu\in\MM(T)$ (Proposition \ref{prop:chaikaex}, see also
Theorem \ref{thm:chaikaex}).\vsp{1.5}

The following questions are open:
\begin{question}\label{q:fin}
    Under the conditions of Proposition \ref{thm:ergiets}, the measure $\phi^{*}(\mu,\nu)$ 
    must be atomic:  $\phi^{*}(\mu,\nu)=\delta_c$ with some $c\in [0,\infty]$. May
    $c$  be a finite positive number, $0<c<\infty$?                          
\end{question}
Note that the metric on $\ui=[0,1)$ in the above question is assumed to be the standard
one, otherwise the answer is `yes' (a counterexample is possible on the basis of 
Remark \ref{rem:gm}, page \pageref{rem:gm}). Note that 
$\phi^*(\mu,\nu)\in\{0,\infty\}$  holds if the measure $\nu$ is decisive, e.\,g. if 
$\nu$ is absolutely continuous relative to $\lam$ (see Corollary 
\ref{cor:conn1}, page \pageref{cor:conn1}).
\begin{question}\label{q:sym} Does the equality $\phi^*(\mu,\nu)=\phi^*(\nu,\mu)$ need    \mbc{q:sym\\2}
       to hold (perhaps, under some  conditions on the scale sequence $\bfs$ 
       and the measures  $\mu,\nu\in\MM(T)$)? 
\end{question}
A map $T\colon\ui\to\ui$ is called $\alp$-collapsing if $\phi_\alp(x,y)=0$ {\em for all}   
$x,y\in\ui$.                                                                              \mbc{q:iet0\\3}
We can prove that there are no $1$-collapsing measurable maps with 
$\MM(T)\neq\emptyset$. Using the axiom of choice we can construct a $1$-collapsing 
(non-measurable) map. 
\begin{question}\label{q:sym} Does there exist a $1$-collapsing measurable map
$T\colon\ui\to\ui$?         \mbc{WHAT?\\NO ?}
(Such a map $T$ cannot have invariant measures: \mbox{$\MM(T)=\emptyset$}).
\end{question}

    In fact, we know of no  example of a  $1$-collapsing \mmps\ $(\ui,T,\leb)$            \mbc{BEGIN\\collapsing}
(on the unit interval).\vsp2

All minimal  $2$-IETs  (equivalently, irrational rotations)  are  $\alp$-collapsing       \mbc{1.85\\att}                       
for all  $0<\alp<1$  (because the sets
$S_q=\{x+i\alpha\text{ (mod 1)}\mid 0\leq k<q\}\subset \ui$\,     are $2/q$-dense     
in $\ui$ for infinitely many  $q\in\N$;  in particular, for all denominators  $q=q_n$
of the convergents  for $\alpha$.  No $2$ or $3$-IET  is  $1$-collapsing (\cite{Tseng2}),      \mbc{refrnce?}
 and it is conjectured by the authors that the answer to          \mbc{END\\collapsing}
Question \ref{q:sym} is  negative.                                                       
\vsp{1.5}

\subsection{Results on Proximality Gauges}                           
Recall that a $3$-IET is determined by a pair $(L,\pi)$  where  $L\in\R^3$  is a 
probability vector and $\pi\in S_3$ is a permutation on the set  \{1,2,3\}. 
In what follows, the implied permutation for $3$-IETs is assumed to be
$\pi=(3\ 2\ 1)$  (reversing the order of the exchanged intervals).                        \mbc{1.9}

We show that that for ``generic" $3$-IETs\, $(\ui,T)$\, and all  $\alp>0$, the relation
\[
    \psi_\alp(x,y)=\liminf_{n\to \infty}\limits n^\alp\,d(T^n(x),T^n(y))
    \equiv\infty\pmod{\lam\ttimes\lam}
\]
\vspm1 \noindent holds  (see \eqref{eq:psia}).                                    
\begin{thm}\label{thm:3psia}
For Lebesgue almost all \ $3$-IETs  (with $\pi=(3\ 2\ 1)$),                                \mbc{thm:3psia}
$\psi^{*}_\alp(\lam,\lam)=\delta_\infty$, for all\, $\alp>0$.
\end{thm}
The result  is quite surprising, especially when compared 
with Corollary  \ref{cor:phi1lam} (on page \pageref{cor:phi1lam}).
Note that ``most'' \mbox{$3$-IETs} $(\ui,T)$\, are weakly mixing (\cite{KaS}), so that\,
\mbox{$\liminf_{n\to \infty}\limits\, d(T^n(x),T^n(y))\equiv0\!\!\pmod{\lam\toto\!}$}
(since\, $T\times T$\, is ergodic).
 
Theorem \ref{thm:3psia} can be restated in the following way.

\begin{cor}\label{cor:3psia}
For Lebesgue almost all \ $3$-IETs\, $T$ (with $\pi=(3\ 2\ 1)$),                            \mbc{cor:3psia}
the proximality constant $C_{\psi}(T)$  vanishes.
\end{cor}

Proposition \ref{proxiet} shows that the proximality constant of any IET is less than or equal to $\frac 1 2$.
In Section \ref{sec:tau} we introduce the notion of the $\tau$-entropy of an IET 
and show that its value provides an upper bound 
for its proximality constant. $\tau$-entropy is connected to the convergence of ergodic averages for the characteristic function of subintervals (Theorem \ref{thm:om2tau}).\vsp1

\subsection{Absence of Topological Mixing for $3$-IETs}  
We prove that no  $3$-IET  can be strongly topologically mixing 
(Theorem \ref{thm:3nomix}, page \pageref{thm:3nomix}).  
Note that the second author (J.~Chaika)  has recently established the existence of  
topologically mixing 4-IETs \cite{chaika3}. It is known that no  IET  
is (measure-theoretically) strongly mixing \cite{nomixing}  and that 
``many'' $3$-IETs  are weakly mixing \cite{KaS}.                                         \mbc{1.95}

\section{\large\bf Contact Gauges and Decisive Measures}\label{sec:dec}                  \mbc{sec:dec}        
Given a sequence $\bfx=\{x_n\}_1^\infty$ in a metric space  $(X,d)$ and a scale sequence  
$\bfs=\{s_n\}_1^\infty$, by the contact gauge (of a sequence $\bfx$ relative to a scale      \mbc{2.0}
sequence $\bfs$) we mean the map $\omega\colon X\to[0,\infty]$ defined as follows:
\begin{align}\label{eq:contact}  
    \omega(x)=\Omega(x,\bfs,\bfx)=\liminf_{n\to \infty}\limits 
    \, s_n\,d(x_n,x) \hsp9 (\text{{\small the contact gauge}}).
\end{align}
\begin{defi}\label{def:dec}
Let a metric space $(X,d)$ be given. A measure $\mu\in\MM(X)$  is called decisive        \mbc{def:dec}
if for every sequence $\bfx=\{x_k\}_1^\infty$ in $X$ and for every scale sequence  
$\bfs=\{s_n\}_1^\infty$  the contact gauge $\omega(x)=\Omega(x,\bfs,\bfx)$  is 
$\mu$-extreme (i.\,e., the equality $\mu(\omega^{-1}(\{0,\infty\}))=1$ holds).
\end{defi}
Thus a measure $\mu\in\MM(X)$ on a metric space $X$ is decisive if all contact
gauges on it are extreme (whatever the choices for $\bfx$ and $\bfs$ are made).
\begin{prop}\label{prop:zinfui}
   The Lebesgue measure on the unit interval $\lam\in\MM(\ui)$ is decisive.              \mbc{prop:zinfui\\2.1}
\end{prop}
There are many examples of decisive measures (absolutely continuous measures on open   
subsets of\, $\R^n$ and smooth manifolds, natural measures on self similar fractals       \mbc{\large ADD}
etc.,  see \cite{BC}). Since the emphasis in the present paper is on the unit interval,
Proposition \ref{prop:zinfui} is confined to this case. 

   Denote by  $B_a(\delta)=\{x\in\ui\mid \fracpart{x-a}<\delta\}$  the open ball of       \mbc{eq:ballui}
radius $\delta>0$ around $a\in\ui$.  (Recall that\, 
$\fracpart{s}=\min_{n\in\Z}\limits |s-n|$\, stands for the distance of\, $s\in\R$\,
to the closest integer).                      
\mbc{proof 888}
\begin{proof}
We have to prove that every contact gauge\,  
$\omega(x)=\liminf_{n\to\infty}\limits s_n |x_n-x|$\, on the unit interval is 
$\lam$-extreme, i.\,e.\ that $\lam(E)=0$  where  $E=\omega^{-1}((0,\infty))$.

For $a>0$ and $N\in\N$ denote\,  $E(a)=\omega^{-1}((a,2a))$\,  and\,  
$E(a,N)=\{x\in E(a)\mid\!\inf_{n>N}\limits s_n|x_n-x|>a\}$. 

It would suffice to prove that\, $\lam(E(a,N))=0$\, for all $a>0$ and $N\in\N$
because the set $E$  is the union of a countable family of sets  $E(a,N)$.
Indeed, one verifies that $E=\bigcup_{k\in\Z}\,E((3/2)^k)$  and that 
$E(a)=\bigcup_{N\in\N}\,E(a,N)$ for all $a>0$.

Let $a>0$ and $N\in\N$ be fixed. Select a point $x\in E(a,N)$. 
Observe that for all $n$  in the infinite set  $S =\{n\in \N\colon s_n|x_n-x|<2a\}$ 
the inclusion $B_{x_n}(\frac a{s_n})\subset B_x(\frac{3a}{s_n})$ holds. 
(The set $S$ is infinite because $x\in E$).
It follows that the set\,  
$\big\{n\in\N\mid B_{x_n}(\tfrac a{s_n})\subset B_x(\frac {3a}{s_n})\,\text{ and }\,
  E(a,N)\cap B_{x_n}(\frac a {s_n})=\emptyset\big\}\,
$
is also infinite because it contains the set $\{n\in S\mid n>N\}$.

Taking in account that $\lam(B_{x_n}(\frac a {s_n}))=\frac13\lam(B_x(\frac {3a}{s_n}))$
and that $\lim_{n\to\infty}\limits\frac a{s_n}=0$, we conclude that\, $x$\, cannot be a 
(Lebesgue) density point for the set $E(a,N)$. 
Since $x\in E(a,N)$ is arbitrary, $\lam(E(a,N))=0$ (by  the Lebesgue density theorem),
and the proof of the proposition is complete. \vspm3
\end{proof}

Now Proposition  \ref{thm:genlam} becomes (in view of Proposition~\ref{prop:zinfui}) 
a special case of the following theorem.
\begin{thm}\label{thm:gendec}
Let $\bfs=\{s_n\}_1^\infty$ be a scale sequence, let $(X,T)$ be a metric system and       \mbc{thm:gendec}
let $\mu,\nu\in\MM(X)$ (not necessarily $T$-invariant). Assume that $\nu$  is decisive
(in $X$). Then\, $\hat\phi(\mu,\nu)\subset\{0,\infty\}$, i.e. 
$\phi^*(\mu,\nu)=c\delta_{\uvm{.6}0}+(1-c)\delta_{\uvm{.6}\infty}$, 
for  some $0\leq c\leq 1$.
\end{thm}
\begin{proof}
For every fixed $x\in X$, 
the values 
$
\phi(x,y)=\liminf_{n\to\infty}\limits s_n d(T^nx,y)=\Omega(y,\bfs,\{T^n x\}_1^\infty)
$
lie in $\{0,\infty\}$  for $\nu$-a.\,a.  $y\in X$  since  $\nu$  is decisive
(Definition \ref{def:dec}). (Equivalently, $\hat\phi(\delta_x,\nu)\subset\{0,\infty\}$).
Since  $\phi(x,y)\in\BB(X)$, Fubini's theorem applies to conclude that 
$\phi(x,y)\in\{0,\infty\}$, for $(\mu\ttimes\nu)$-a.\,a. pairs $(x,y)\in X^2$,
completing the proof of the theorem.\vspm3
\end{proof}
Note that in the special case when $\bfs$  is a two-jumpy scale sequence (e.\,g., if    \mbc{2.2}
$\bfs$ is a power scale sequence: $s_n=n^\alp$, for some $\alp>0$), no assumption on 
``decisiveness'' of $\nu$ is needed because of Theorem \ref{thm:gsoinf}.

In the next section we collect basic facts on the gauges\, $\phi, \psi, \rho$\,  which hold
for general metric systems.  This sets a framework for the proof and comparison with
more specific results on IETs, and also motivates these results.

\section{\bf \large From Quantitative to Connectivity Gauges in General Systems}\label{sec:general}                                    \mbc{sec:general}

In this section we present general results, mostly on the connectivity gauges. 
The motivation comes partially from \cite{bosh3} and \cite{WBe}
where Quantitative Recurrence study has been initiated. 

\subsection{Quantitative Recurrence Results, Review}
In particular, a very general lower bound for the recurrence 
constant is found:
\begin{thm}\label{thm:bosh3} 
\cite{bosh3}  Let $(T,X,\mu,d)$ be an m.m.p system and $\alpha$ be the Hausdorff dimension of $\mu$ 
(with respect to the metric $d$). 
Then $C_{\rho}\geq \frac 1 \alpha$.
\end{thm}
An analogous version of this for the connectivity gauge is Proposition \ref{Hdim}
(see also Remarks \ref{rem:Hdim con2} and \ref{rem:Hdim prox} in the end of the section).

It is known that for the \mmps s on the unit interval the $1$-recurrence gauge\, $\rho_1$\, 
must be finite almost everywhere; moreover, it is bounded by $1$ if $\mu=\lam$.

\begin{prop}[\cite{bosh3}]\label{prop:recui1}
For every \mmps\  $(\ui,T,\mu)$, the inequality $\rho_1(x)<\infty$ holds for $\mu$-a.\,a.
$x\in\ui$. Moreover, in the special case $\mu=\lam$ the gauge $\rho_1$ is essentially 
bounded: $\rho_1(x)\leq1\pmod{\lam}$.
\end{prop}

Note that the inequality $\rho_1(x)\leq1$ in the above proposition can be replaced 
by $\rho_1(x)\leq1/2$ (unpublished), and the authors conjecture that the optimal constant 
is $\frac1{\sqrt{5}}$.

Theorem \ref{thm:bosh3} (and its versions) have been used in the study of recurrent 
properties of specific dynamical systems (\cite{bosh3}, \cite{BoKo}, \cite{BGI},
\cite{chaika1}, \cite{RS} and \cite{Shk02b}); there are several refinements of it,
often under the additional assumptions (see \cite{BaSa}, \cite{Bar}, \cite{Ga}, 
\cite{g and k}, \cite{Shk02}), \cite{Mo}).

Shkredov in \cite{Shk05}, \cite{Shk07a}, \cite{Shk07b} obtained quantitative pointwise 
multiple recurrence results  (of Szemeredy-Furstenberg type) using Gower's quantitative 
estimates (\cite{Go}, \cite{Shk05}) and the approach in \cite{bosh3}. Kim \cite{KimQR} has recently 
extended Theorem 1 to arbitrary group actions.


\subsection{Quantitative Connectivity Results}
Given a metric system $(X,d,T)$, the dilation gauge \ 
$D_T\colon X\to [0,\infty]$ \ is defined by the formula
\begin{equation}\label{eq:dt}
D_T(x)=
\begin{cases}
\limsup_{\substack{y\to x\\ y\in X}}\limits\,
\uv{1.1}{$\frac{d(T(y),T(x))}{d(y,x)}$}, &
\text{ if }\ x\, \text{ is a limit point of  $X$},\\
0, & \text{ if }\ x\, \text{ is an isolated point in  $X$.}
\end{cases}
\end{equation}

Let $\mu\in\MM(X)$ be a measure. The map  $T$  is said to be a local contraction 
$\!\!\!\pmod{\mu}$ if $D_T(x)\leq1$ for $\mu$-a.\,a. $x\in X$. For example, every IET is 
a local contraction  $\!\!\!\pmod{\mu}$ for every continuous measure  $\mu\in\MM(\ui)$.

The map $T$  is said to be locally Lipshitz $\!\!\!\pmod{\mu}$ if $D_T(x)<\infty$, for 
$\mu$-a.\,a. $x\in X$.  For example, a piecewise monotone map $T:\ui\to\ui$ must be 
locally Lipshitz  $\!\!\!\pmod{\lam}$.\vsp1

For  $x\in X$ and $\mu\in\MM(X)$, we abbreviate 
$\phi^*(\mu,\delta_x)=\phi^*(\mu\ttimes\delta_x)$  to just  $\phi^*(\mu,x)$                \mbc{2.25}
(for notation see \eqref{eqs:fall} and \eqref{eqs:gauges}).
In the similar way,  $\phi^*(\delta_x,\mu)$  is abbreviated to  $\phi^*(x,\mu)$.
(Recall that $\delta_x\in\MM(X)$ stands for the atomic measure supported 
by the point $x\in X$).

Recall that a scale sequence is called nice if it is both two-jumpy and bounded-ratio
(Definition~\ref{def:scalep}, page~\pageref{def:scalep}).
\begin{prop}\label{prop:connectivity} 
Let a metric system  $(X,d,T)$ and a scale sequence  $\bfs=\{s_n\}_1^\infty$ 
be given. Then:                                                                            \mbc{prop:con\\nectivity}
\begin{itemize}
\item[(1)]  Assume that the scale sequence $\bfs$\, is either monotone or steady. 
   Then for every\, \mbox{$T$-ergodic} measure $\mu\in\MM(T)$ and a point $y\in X$, the 
   measure $\phi^*(\mu,y)=\phi^*(\mu\times\delta_y)\in\MM([0,\infty]$ is atomic.
   Moreover, if\,~\bfs\, is two-jumpy then $\phi^*(\mu,y)\in\{\delta_0,\delta_\infty\}$.
   \vsp{1.5}
\item[(2)]  
   Assume that $\bfs$\, is a steady scale sequence and that $T$ is a local contraction      \mbc{2.3}
   $\!\!\pmod{\nu}$, for some  \mbox{$T$-ergodic} measure $\nu\in\MM(T)$.
   Then for every  $x\in X$, the measure $\phi^*(x,\nu)\in\MM([0,\infty]$  is atomic. 
   \vsp{1.5}
\item[(3)]  
   Assume that:

\hsp5  (a)\ $\bfs=\{s_n\}_1^\infty$ is a nice scale sequence,

\hsp5  (b)\ The measure $\nu\in\MM(T)$ is decisive (Definition \ref{def:dec}, 
                 page \pageref{def:dec}) and $T$-ergodic,
                 
\hsp5  (c)\ $T$ is locally Lipshitz$\pmod{\nu}$.

\hspm5  Then $\phi^{*}(x,\nu)\in\{\delta_0,\delta_\infty\}$ 
        for every  $x\in X$. \vsp{1.5}
\end{itemize}
\end{prop}
\begin{proof}[\bf\em Proof of Proposition \ref{prop:connectivity}] \ 
{\em Proof of\, \em(1).}
Observe that for all $x,y\in X$ the inequality 
\[
\phi(Tx,y)=\liminf_{n\to\infty} s_n\,d(T^n(Tx),y)=\liminf_{n\to\infty}\,
\frac{s_n}{s_{n+1}} \left(s_{n+1} d(T^{n+1}x,y)\right)\leq\phi(x,y)
\]
holds. For a fixed $y$, the ergodicity of $T$ implies that 
$\phi^{*}(\mu,y)\in\MM([0,\infty])$ is an atomic measure. Finally, 
if $\bfs=\{s_n\}_1^\infty$ is two-jumpy, the EGT (Theorem \ref{thm:egt}) implies that 
$\phi^{*}(\mu,y)\in\{\delta_0,\delta_\infty\}$.\vsp2

\noindent{\em Proof of\, \em(2).} Fix $x\in X$. Then for $\nu$-a.\,a. $y\in X$
we have
\begin{eqnarray}
\phi(x,Ty)=\liminf_{n\to\infty} s_n\,d(T^nx,Ty)&=&
\liminf_{n\to\infty}\ \frac{s_n}{s_{n-1}}\cdot\frac{d(T(T^{n-1}x),Ty)}{d(T^{n-1}x,y)}
\cdot\left(\,s_{n-1}\,d(T^{n-1}x,y)\right)\label{eq:phixty}\\[1.5mm]
&\leq& D_T(y)\phi(x,y)\leq \phi(x,y).\notag
\end{eqnarray}

Since $\nu$ is $T$-ergodic, $\phi^{*}(x,\nu)\in\MM[0,\infty]$ is atomic.\vsp2
\noindent{\em Proof of\, \em(3).} Fix $x\in X$. Now the inequality \eqref{eq:phixty}
takes the form
$
\phi(x,Ty)\leq \kappa D_T(y)\phi(x,y)
$
where $\kappa=\sup_{n\geq1}\limits \frac{s_{n+1}}{s_n}$.
Note that $\hat\phi(x,\nu)\subset\{0,\infty\}$, in view of the decisiveness of~$\nu$.      \mbc{2.35}

Since $T$  is locally Lipshitz$\pmod{\nu}$, it follows that 
$\phi(x,Ty)\leq\phi(x,y)$,  for $\nu$-a.\,a.\ $y\in X$. 
Now, the ergodicity of\, $\nu$\,  implies that $\phi^{*}(x,\nu)\in\MM([0,\infty])$  is 
atomic.  Therefore  $\phi^{*}(x,\nu)\in\{0,\infty\}$.

This completes the proof of Proposition \ref{prop:connectivity}.
\end{proof}
\vsp1
\begin{lem}\label{lem:singleton}
Let  $(X,\mu)$ and $(Y,\nu)$ be two \mms s and assume that\, $\phi$\,  is a gauge on         \mbc{lem:singleton}
the product  $P=X\times Y$  (i.\,e., a measurable map  
$\phi\colon X\times Y\to[0,\infty]$). Assume that the following two conditions hold:
\begin{itemize}
\item for $\mu$-a.\,a. $x\in X$ the measures\,  $\phi^{*}(x,\nu)$ are atomic;
\item for $\nu$-a.\,a. $y\in Y$\, the measures\,  $\phi^{*}(\mu,y)$ are atomic.
\end{itemize}
Then the measure $\phi^{*}(\mu,\nu)$  is atomic too.
\end{lem}
\begin{proof} \
Set $f(x)=\int_Y \phi(x,y)\,d\nu(y)$, for $x\in X$. 
Set $g(y)=\int_X \phi(x,y)\,d\mu(x)$, for $y\in Y$.

Then \
$
f(x)=\phi(x,y)=g(y) \pmod{\mu\ttimes\nu},
$
and the claim of the lemma becomes obvious by Fubini's Theorem.
\end{proof}
\begin{thm}\label{thm:conn1}
Let a metric system $(X,d,T)$  and a steady scale sequence $\bfs=\{s_n\}_1^\infty$        \mbc{thm:conn1}
be given. Let $\mu,\nu\in\MM(T)$ be two $T$-ergodic measures. Assume that $T$ is a 
local contraction $\!\!\pmod{\nu}$. Then the measure $\phi^*(\mu,\nu)$  is atomic.
\end{thm}
\begin{proof}
Since \ $\bfs$ \ is steady, it follows from Proposition \ref{prop:connectivity}(1) 
that the measures $\phi^{*}(\mu,y)$ are atomic, for every $y\in Y$.
On the other hand, by Proposition \ref{prop:connectivity}(2) the measures
$\phi^{*}(x,\nu)$ are atomic too, for every $x\in X$. 

It follows from Lemma~\ref{lem:singleton} that 
the measure $\phi^{*}(\mu,\nu)$ is atomic.
\end{proof}
  Note that under the conditions of Theorem \ref{thm:conn1} one cannot claim that         \mbc{2.4}
$\phi^{*}(\mu,\nu)\in\{\delta_0,\delta_\infty\}$  (see Example \ref{ex:gm}).
Some additional assumptions are needed.
\vsp2

\begin{cor}\label{cor:conn1}
Under the conditions of Theorem \ref{thm:conn1}, assume that (at least) one of the 
following conditions holds:
\begin{enumerate}
\item[\em(1)] The measure $\nu$  is decisive (see Definition \ref{def:dec},
              page \pageref{def:dec}).
          
\item[\em(2)] The scale sequence  \bfs  \ is two-jumpy.
\end{enumerate}
Then $\phi^{*}(\mu,\nu)\in\{\delta_0,\delta_\infty\}$, i.\,e.\ the gauge $\phi$ is
$\mu\ttimes\nu$-trivial.
\end{cor}
\begin{proof}
By Theorem \ref{thm:conn1}, $\phi^{*}(\mu,\nu)=\delta_{c}$, for some $c\in[0,\infty]$.      \mbc{2.42}
Assuming (1), we have  $\phi^{*}(x,\nu)\subset\{\delta_0,\delta_\infty\}$, 
for all $x\in X$. It follows that $\hat\phi(\mu,\nu)\subset\{0,\infty\}$.  
Thus $c\in\{0,\infty\}$.
Assuming (2), the inclusion $\hat\phi(\mu,\nu)\subset\{0,\infty\}$ holds 
in view of Theorem \ref{thm:gsoinf}. Thus $c\in\{0,\infty\}$.
The proof is complete.
\end{proof}

\begin{thm}\label{thm:conn2}
Let $(X,d,T)$ be a metric system and let $\mu,\nu\in\MM(T)$ be two $T$-ergodic measures.   \mbc{thm:conn2}
Assume that\, $T$\, is locally Lipshitz $\!\!\pmod{\nu}$ and that 
$\bfs=\{s_n\}_1^\infty$ is a nice scale sequence.  Then                                    \mbc{2.43}
$\phi^{*}(\mu,\nu)\in\{\delta_0,\delta_\infty\}$ 
(i.\,e., the gauge $\phi$ is $\mu\ttimes\nu$-trivial).
\end{thm}
\begin{proof}
The scale sequence $\bfs$ is two-jumpy, hence monotone.                                   \mbc{2.44}
Proposition~\ref{prop:connectivity}(1) applies to conclude that the set 
$\phi^{*}(\mu,y)$  is atomic for all $y\in X$.  Moreover, the inclusion
\begin{equation}\label{eq:shortdot5}
\phi^{*}(\mu,y)\in\{\delta_0,\delta_\infty\} \quad \text{ (for all }y\in X),
\end{equation}
holds in view of the EGT (Theorem \ref{thm:egt}).  

To conclude the proof we only need to show that $\phi^{*}(x,\nu)$  is atomic
for $\mu$-a.\,a.\ $x\in X$ (see Lemma \ref{lem:singleton}).
This task is accomplished by exhibiting the set
\[
X'=\{x\in X\mid \phi(x,Ty)\leq \phi(x,y) \quad \text{for }\nu\text{-a.\,a.\ }y\in X\},
\]
with the following two properties:\\[2mm]
\hsp{20}(p1)  $\mu(X')=1$.\\[.5mm]
\hsp{20}(p2)  For every $x\in X'$, $\phi^{*}(x,\nu)$ is atomic.\vsp3

Note that the relation \eqref{eq:shortdot5}  implies that
\begin{equation}\label{eq:dotphimunu}
\phi^{*}(\mu,\nu)\subset\{\delta_0,\delta_\infty\}.
\end{equation}
Just as in the proof of 
Proposition \ref{prop:connectivity}(3), the inequality   
\[
\phi(x,Ty)\leq \kappa D_T(y)\phi(x,y) \pmod{\nu(y)}
\]
holds for every $x\in X$. It follows that                                                  \mbc{2.45}
\begin{equation}\label{eq:philong}
\phi(x,Ty)\leq \kappa D_T(y)\phi(x,y) \pmod{\mu\times\nu}.
\end{equation}

In view of \eqref{eq:dotphimunu} and the fact that 
$D_{T}(y)<\infty$  for $\nu$-a.\,a.\ $y\in X$, the relation \eqref{eq:philong} implies 
\[
\phi(x,Ty)\leq \phi(x,y) \pmod{\mu\times\nu},
\]
validating the property (p1). The $T$-ergodicity of $\nu$ implies (p2). 
The proof is complete.
\end{proof}

An important special case of Theorem \ref{thm:conn2} when \ 
$\bfs=\bfs_{\alp}=\{n^\alp\}_{1}^{\infty}$ is a power scale sequence is outlined in        \mbc{2.5}
the following theorem.  Recall that the $\alp$-connectivity gauge (the connectivity 
gauge with $\bfs=\bfs_{\alp}$)  takes the form 
$\phi_{\alp}(x,y)=\liminf_{n\to\infty}\limits n^\alp\,d(T^nx,y)$ 
(see \eqref{eqs:gaugesa}).
\begin{cor}\label{thm:connpower}
Let $(X,d,T)$ be a metric system and let $\mu,\nu\in\MM(T)$ be two $T$-ergodic measures.     \mbc{thm:connpower}
Assume that\, $T$\, is locally Lipshitz $\!\!\pmod{\nu}$. Then all power connectivity
gauges $\phi_{\alp}$ are $\mu\ttimes\nu$-trivial (i.\,e., for every $\alp>0$ 
the inclusion $\phi^{*}_{\alp}(\mu,\nu)\in\{\delta_0,\delta_\infty\}$ holds).
\end{cor}
\begin{proof}
Follows from Theorem \ref{thm:conn2}  because power scale sequences $\bfs=\{s_{n}\}$  
are nice.
(Alternative argument: It also follows from Corollary \ref{cor:conn1}(2) because
power scale sequences $\bfs=\{s_{n}\}$ are both two-jumpy and steady).
\end{proof}
\begin{rem} A.\ Quas has constructed a $\lam$-ergodic transformation of $\ui$ such that 
$\phi_1^*(\leb)=\delta_{\infty}$ (personal communication).
\end{rem}

\begin{prop}\label{Hdim}
Let $(T,X,\mu,d)$ be an m.m.p system and $\alpha$ be the Hausdorff dimension of $\mu$ (with respect to the metric $d$). Then $C_{\phi} \leq \frac 1{\alpha}$. 
\end{prop}
\begin{proof} Let $s<C_{\phi}(\mu)$. So $\mu(\LS B(T^ix,\frac 1 {i^s}))=1$ for $\mu$ almost every $x$. Because $\boundsum (\frac {1} {n^s})^{\frac 1 s +\epsilon}$ converges for any $\epsilon>0$, and $\underset{i=t}{\overset{\infty}{\cup}}B(T^ix,\frac 1 {i^s})$ covers $\LS B(T^ix,\frac 1 {i^s})$ for all $t$, the Hausdorff dimension of $\LS B(x,\frac 1 {n^s}) \leq \frac 1 s$. Thus
$\mu$ assigns full measure to a set of $H_{dim} \leq \frac 1 s$. Because $s$ can be arbitrarily close to $C_{\phi}$ the proposition follows.
\end{proof}
\begin{rem} \label{rem:Hdim con2}
One can define the upper connectivity constant by ${\inf \{\alpha: \phi^*_{\alpha}(\mu,\mu)(0)>0\}}$. One can define the lower Hausdorff dimension of a measure to be the infimum of the Hausdorff dimensions of all Borel sets of positive $\mu$ measure. By a similar argument one obtains that the upper connectivity constant is less than the reciprocal of the lower Hausdorff dimension of $\mu$. It is easy to construct H\"{o}lder examples where the upper and lower connectivity constants are different (these examples can show that the Lipshitz assumption in Theorem \ref{thm:conn2} and Corollary \ref{thm:connpower} is important).
\end{rem}
\begin{rem} The inequality in this proposition is sharp. For every $n\geq2$ there are 
(Lebesgue) ergodic rotations of the $n$-torus that have connectivity constant 0. 
(For $n=1$, the connectivity constant is always $1$).
Lebesgue measure on the $n$-torus has Hausdorff dimension $n$.
\end{rem}
\begin{rem}\label{rem:Hdim prox} 
An analogous proof shows that $C_{\psi} \leq \frac 1{ H_{dim}(\mu)}$. (Compare with
Proposition \ref{Hdim} and Theorem \ref{thm:bosh3}).
\end{rem}
\section{\bf \large Two Results on the Proximality Constant} \label{sec:prox}
\begin{prop} \label{proxiet} Let $T$ be an IET then $C_{\psi}(T)\leq \frac 1 2$.
\end{prop}
\begin{proof} 
Consider $U_n=\{(x,y): d(T^nx,T^ny)<\min (d(T^{n-1}x,T^{n-1}y),\frac1{n^{c}})\}.$ 
If $x$ appears in the first coordinate then $d(x,\delta_i)< \frac 1 {n^c}$  for some 
discontinuity $\delta_i$. The measure of such points is $\frac{2(r-1)} {n^c}$ (if $T$ is 
an $r$-IET). For a fixed\, $x$, the inequality 
${\lambda({x}\times[0,1)\cap S_n)\leq \frac{2}{n^c}}$ holds. Therefore 
$\lambda^{(2)}(U_n)<\frac{4(r-1)}{n^{2c}}$, and \, $\sum \lambda^{(2)}(U_n)$ converges
if  $c>\frac 1 2$. By the Borel-Cantelli Theorem, $C_{\psi}(T)\leq \frac 1 2$.
\end{proof}
\begin{prop} \label{proxlinrec} Let $S:X \to X$ be a linear recurrent dynamical system 
and $\bar{d}$ be the standard metric, 
$\bar{d}(\bar{x},\bar{y})= \frac 1 {{\min}\{i:\,x_i \neq y_i\}}$ and $\mu$ its unique ergodic 
measure. Then $C_{\psi}(X,\mu,\bar{d})\leq \frac 1 2$.
\end{prop}
We review some facts about linear recurrent transformations. 
Linear recurrent transformations are uniquely ergodic. If $\mu$ is the unique invariant measure then 
$\frac 1 {Cn}<\mu(x_1,x_2,x_3,...,x_n,*)<\frac C n$ where $C$ is a constant depending only 
on the transformation ($(x_1,x_2,...,x_n,*$ denotes all words beginning $(x_1,x_,...x_n)$). They have linear block growth.
Systems of linear block growth have bounded difference between the number of allowed 
$k$ blocks and allowed $k+1$ blocks \cite{lin block}.
\begin{proof}
 Consider $U_n=\{(x,y):  d(S^n\bar{x},S^n\bar{y})<\min (d(T^{n-1}\bar{x},T^{n-1}\bar{y}), 
 \frac 1 {n^{c}})\}.$ If $(\bar{x},\bar{y}) \in U_n$ then $x_{n-1}\neq y_{n-1}$ and 
 $x_{n+k}=y_{n+k}$ for $0 \leq k <n^c$. Let $K$ be the largest difference between the number 
 of allowed $k+1$ blocks and $k$ blocks.  Then $\mu^{(2)}(U_n)\leq \frac{K^2C}{n^{2c}}$. 
\end{proof}
\section{\bf \large Proof of Theorem \ref{thm:phi1mu}}                                                  \mbc{new\\section}
The following lemma (implicit in \cite{nomixing}) will be used in the proof.
\begin{lem}\label{lem:1} 
Let  $(\ui,T)$  be a minimal $r$-IET and let $\mu\in\MM(T)$ be an invariant measure.          \mbc{lem:1}
Then, for every $\eps>0$, there are a subinterval $J\subset\ui$ and an integer $N\in\N$ 
such that: 
\begin{enumerate}
\item $\mu(J)< \epsilon$;
\item $\mu\big(\bigcup_{n=0}^{N-1}\limits T^n(J)\big)\geq \frac1r $;
\item the sets $J, T(J),...,T^{N-1}(J)$ are pairwise disjoint subintervals of\, $\ui$.
\end{enumerate}
\end{lem}
\begin{proof} Pick a subinterval $I\subset\ui$, $0<\mu(I)<\epsilon$,  such that             \mbc{proof\\2.6}
the induced map  $T'$  on $I$ is an  $s$-IET, for some  $2\leq s\leq r$.  
Let  $I_k\subset I$,  $1\leq k\leq s$,  be the subintervals of  $I$  exchanged 
by  $T'$,     $T'(I_k)=T^{N_{k}}I_k.$

By minimality, the images $T^n(I_k)$ cover $\ui=[0,1)$ before returning to $I$.
More precisely,  the family of subintervals
$\big\{T^n(I_k) \mid1\leq k\leq s, 0\leq n\leq N_k \big\}$
partition the interval\,  $\ui=[0,1)$. 

Thus\, $\mu\Big(\bigcup_{n=0}^{N_k} T^n(I_k)\Big)\geq\tfrac1s\geq\tfrac 1r$,\,
for some  $k, 1\leq k\leq s.$
To complete the proof of the lemma, one takes  $J=I_k$\, and $N=N_k$.
\end{proof}
\begin{proof}[Proof of Theorem {\em \ref{thm:phi1mu}}]
Given an IET $T$ with a  $T$-ergodic measure $\mu\in\MM(T)$,  we have to show  that         \mbc{proof\\thm:phi1mu}
\mbox{$\phi^{*}_{1}(\mu,\mu)=\delta_0$}. Since $\mu$ is ergodic, it is supported by a 
minimal component $K$  which is either a finite set (permuted cyclically by\, $T$), or 
a finite collection of subintervals of $\ui$.  (It is well known that the domain of an
IET admits partition into a finitely many periodic and continuous components; each such 
component is a finite union of subintervals. For an algorithmic way of finding this 
decomposition in certain situations see e.\,g.\ \cite{rank2}).

In the first case $\phi^{*}_{1}=\delta_0$ is immediate (because then 
$\phi(x,y)=0$ for all $x,y\in K$).

In the second case, $T|_K$  is itself a minimal IET (after making the subintervals of       \mbc{2.7}
$K$ adjacent to each other and rescaling). 
Thus without loss of generality we may assume that the original $T$ is a minimal IET
(and then $\mu$ is continuous).

We first claim that the gauge $\phi_{1}(x,y)=\liminf_{n\to\infty}\limits\, n|T^n(x)-y|$\,     
is\, $\mu\ttimes\mu$-trivial, i.\,e.\ that\, 
$\phi^{*}_1(\mu,\mu)\in\{\delta_0,\delta_\infty\}$.
This follows from Corollary \ref{cor:conn1}(2) (alternatively, from 
Corollary~\ref{thm:connpower}).

It remains to show that $\phi^{*}_1(\mu,\mu)=\delta_\infty$ leads to a contradiction.

For every integer  $k\geq2$, one takes an interval $J_k$ and an integer $N_k$ as in            \mbc{2.75}
Lemma \ref{lem:1} with  $\eps=\eps_{k}=\frac1{3rk}$. 

Then the inequality $N_k|J_k|\geq\frac1r$ implies 
$\frac1{3rk}=\eps_{k}>|J_k|\geq \frac1{rN_k}$ whence $N_k>3k$.  Set\,  
$m_k=\big[N_k/3\big]\geq k$ and define the sets
\[
A_k=\bigcup_{i=0}^{m_k-1}\limits T^i(J_{k}), \qquad
B_k=\!\!\!\bigcup_{i=N_k-m_k}^{N_k-1}\limits\!\!\!\! T^i(J_{k}), \qquad
C_k=A_k\times B_k\subset \ui^2.
\]

We estimate\,
$
\mu(A_k)=\mu(B_k)=m_k\mu(J_k)\geq\frac{m_k}{rN_k}\geq\frac{m_k}{r(3m_k+2)}\geq\frac1{5r}
$\,
whence\,  $(\mu\times\mu)(C_k)\geq\frac1{25r^2}$, for all $k$.

It follows that\,  $(\mu\times\mu)(D)\geq\frac1{25r^2}$\,  where\,
$D=\limsup_{k\to\infty}\limits\, C_k=
       \bigcap_{n\geq2}\limits\!\big(\!\bigcup_{k=n}^\infty\limits C_k\big)\subset\ui^2$.

Observe that if $(x,y)\in C_{k}$  then there is an $m$, $\frac{N_k}3<m<N_k$, such that      \mbc{2.8}
the points $T^m(x)$ and $y$ lie in the same subinterval $T^j(J_{k})$ (for some 
$j\in[N_k-m_k,N_k]$) of length $|J_{k}|<\frac1{N_{k}}$. 

It follows that for all\, $k\geq2$
\[
     \min_{\frac{N_k}3\leq m<N_{k}}\limits\! m\,|T^{m}(x)-y|\leq
     \min_{\frac{N_k}3\leq m<N_{k}}\limits\!\! N_K\,|T^{m}(x)-y|<1, \qquad 
     \text{for }\,(x,y)\in C_k.
\]

Since $N_k>k$, the definition of $D$ and the last inequality imply
\[
\phi_1(x,y)=\liminf_{n\to\infty}\limits n|T^n(x)-y|\leq1,\quad \text{ for }(x,y)\in D.
\]
Now the assumption $\phi^{*}(\mu,\mu)=\delta_{\infty}$ would imply that                      \mbc{2.9}
$(\mu\times\mu)(D)=0$, a contradiction with the earlier conclusion that 
$(\mu\times\mu)(D)\geq\frac1{25r^{2}}$.
\end{proof}

\begin{rem} The above proof and Theorem \ref{thm:phi1mu} extend to many finite rank 
transformations.
\end{rem}

\section{\bf \large Proof of Theorem \ref{thm:noptim}}
Given  a sequence ${\bf s}=(s_k)_1^\infty$ of positive real numbers  such that                \mbc{NEW SEC\\
																						    thm:noptim\\
																						    proof}
\begin{equation}\label{eq:sinf}
\lim_{n\to\infty}\limits s_n/n=\infty,
\end{equation}
we have to construct an irrational\, $\alpha \in \R\!\setminus\! \Q$\, such that\,              \mbc{eq:sinf}
$
\lim_{n\to\infty}\limits s_n\fracpart{n \alpha - y}=\infty,
$
for Lebesgue a.\,a. $y\in\ui$.

For  $a,r\in\R$,  denote   $B(a,r)=\{x\in\ui\mid\fracpart{x-a}<r\}$.                        \mbc{eq:balls}
It is easy to verify that
\begin{equation}\label{eq:balls}
        \lam\big(B(a_1,r)\!\setminus \!B(a_2,r)\big)\leq \fracpart{a_2-a_1}, \quad 
        \text{ for all }\,  a_1,a_2,r\in\R,
\end{equation}
and that  $\lam(B(a,r))=2r$\,  if\,   $0\leq r\leq 1/2$.                                      \mbc{3.0}

We may assume  that the sequence  ${\bf s}=(s_k)$  is non-decreasing.
(Otherwise we replace it  by the sequence   ${\bf s'}=(s'_k)$  where\,  
$s_k'=k\cdot(\min_{n\geq k}\limits\,  s_n/n)$.   The new sequence
${\bf s'}$  is strictly increasing, is dominated by ${\bf s}$  and 
\eqref{eq:sinf}  implies that $\lim_{n\to\infty}\limits s'_n/n=\infty$\,  holds).

The required  $\alp\in\ui\cap(\R\setminus\Q)$  will be determined by its 
continued fraction                                                                             \mbc{eq\\alpcf\\
																							3.1}
\begin{equation}\label{eq:alpcf} 
\alp=\cf{0, a_1, a_2,\ldots }\bydef\frac1{a_1+\frac1{a_2+\ldots}}\qquad  a_k\in\N.
\end{equation}

Recall some basic facts from the theory of continued fractions  \cite{Khinchin}.  
The numbers    $\alp\in\ui\cap(\R\setminus\Q)$  are in the one-to-one 
correspondence with the sequences  $(a_k)_1^\infty$, their continued fraction 
expansions.  The convergents of  $\alp$  are defined by formula
 $\alp_k=p_k/q_k=\cf{0, a_1,\ldots, a_k}$,
where the sequences   $(p_k), (q_k)$  are defined by the initial conditions
$p_0=0$, $q_0=1$, $p_1=1$, $q_1=a_1$  and the recurrent relation                              \mbc{3.2}
$x_k=a_k x_{k-1}+x_{k-2}$,  satisfied by both for  $k\geq2$.

The following inequality (standard in the theory of continued fractions) 
will be used later:                                                                           \mbc{3.3\\ineq:qk} 
\begin{equation}\label{ineq:qk}
\fracpart{\alp q_n}<1/q_{n+1}, \quad n\geq1.
\end{equation}
For an irrational number  $\alp\in\ui$, define
\[
A_{k,c}(\alp)=\{x\in\ui\mid s_n\fracpart{n\alp-x}<c,  \text{ for some } 
n\in[q_k,q_{k+1})\cap\N\},
\]
where  $q_k$  (the denominators of convergents)  are determined by  $\alp$.    
By Borel Cantelli lemma, in order to prove Theorem \ref{thm:noptim},
it would suffice to find an irrational number  $\alp\in\ui$
such that                                                                                   \mbc{3.4\\eq:bc}
\begin{equation}\label{eq:bc}
\sum_{k\geq1}\limits \lam(A_{k,c}(\alp))<\infty,\quad  
\text{ for all } c>0.
\end{equation}  
We exhibit such an $\alp$ explicitly by its continued fraction  
\eqref{eq:alpcf}  by setting                                              \mbc{eq:ak\\3.5}
\begin{equation}\label{eq:ak}
a_k= \max(N_k,3k^2), \quad k\geq1,
\end{equation}
 where  $N_k=\min S_k$\,  and\,
 $S_k=\{m\in\N\mid s_n/n\geq k^4, \text{ for all } n\geq m\}\subset\N$.
 Note that the sets  $S_k$  are not empty in view of the 
 assumption~\eqref{eq:sinf}.  (One can verify that, in fact, 
 the inequalities  $a_k\geq \max(N_k,3k^2)$, 
rather than \eqref{eq:ak},  would suffice for  our task).                                      \mbc{to delete?}

Set  $m_k=\intpart{q_{k+1}/k^{2}}$  
(where $q_k$  are the denominators of the convergents for  $\alpha$
determined by the expansion~\eqref{eq:alpcf}).  Then                                            \mbc{3.6\\ineq:mk}
\begin{equation}
q_{k+1}\geq m_k\geq 3q_k>q_k, \quad \text{for } k\geq1.
\end{equation}
(Indeed, $ q_{k+1}\geq\intpart{q_{k+1}/{k^2}}= m_k\geq 
   \intpart{a_{k+1}q_k/k^2}\geq 
   \intpart{a_{k+1}/k^2}\cdot q_k\geq 3q_k.
$)  

Denote\,  $B_n=B(n\alp,\tfrac c{s_n})$ and set                                                \mbc{3.7}
\[
S_1=\leb\big(\!\!\bigcup_{n=q_k}^{m_k}\limits\! B_n\big), \quad 
S_2=\leb\big(\!\!\bigcup_{n=m_k}^{q_{k+1}}\limits\! B_n\big), \quad
S_3=\leb\big(\!\!\bigcup_{n=q_k}^{2q_k}\limits\! B_n\big) \quad
\text{and}\quad
S_4= \leb\big(\!\!\!\bigcup_{\hsp1n=2q_k}^{m_k}\limits\!\!
B_n\!\!\setminus\! B_{n-q_k}\big).
\]
\noindent Then    
$\leb\big(A_{k,c}(\alp)\big)\leq S_1+S_2$\, and\, $S_1\leq S_3+S_4$.                           \mbc{3.8}
It follows that\, $\leb\big(A_{k,c}(\alp)\big)\leq S_2+S_3+S_4$.

Since  $m_k=\intpart{q_{k+1}/{k^2}}$, the inequalities $m_k>q_k>a_k\geq N_k$  imply
\[
S_2\leq  \frac{2c\,q_{k+1}}{s_{m_k}}< 
\frac{2c\,q_{k+1}}{m_k+1}\cdot\frac{2m_k}{s_{m_k}}
\leq 4ck^2\cdot\frac{m_k}{s_{m_k}}\leq \frac{4ck^2}{k^4}=4ck^{-2}.
\]

An estimate for  $S_3$  follows from the inequalities  $q_k>a_k\geq N_k$:                     \mbc{3.9}
\[
S_3\leq \frac{c\,q_k+1}{s_{q_k}}< (c+1) q_k/{s_{q_k}}\leq (c+1)/{k^4}.
\]

In order to estimate  $S_4$,  we first observe that  
$s_{n-q_k}\leq s_n$  implies
\begin{align*}
\lam(B_n\!\!\setminus\! B_{n-q_n})&=
\lam\big(B(n\alp, c/\!\rbm{.5}{$s_n$})\!\setminus 
B((n-q_k)\alp, c/\!\rbm{.5}{$s_{n-q_k}$}\big)\leq\\
&\leq \lam\big(B(n\alp,c/\!\rbm{.5}{$s_n$})\!\setminus 
B((n-q_k)\alp,c/\!\rbm{.5}{$s_n$}\big),
\end{align*}
whence\,                                                                                      \mbc{4.0}
$
\lam(B_n\!\!\setminus\! B_{n-q_n})\leq \fracpart{\alp q_k}<1/q_{k+1}
$
(see  \eqref{eq:balls} and  \eqref{ineq:qk}).  
Since  $m_k=\intpart{q_{k+1}/{k^2}}$,  we get
\[
S_4\leq \frac{m_k-2q_k+1}{q_{k+1}}<
\frac {m_k}{q_{k+1}}\leq \frac1{k^2}.
\]

It follows that\,   
$\leb\big(A_{k,c}(\alp)\big)\leq S_2+S_3+S_4<(4c+1)/{k^2}+(c+1)/{k^4}$,
so that  \eqref{eq:bc}  holds. 

The proof of Theorem \ref{thm:noptim}  is complete.  
                                                                                            \mbc{4.1- end\\
																			              proof of\\
                                                                                             thm:noptim}

\section{\bf \large Proof of Theorem \ref{thm:chaikaex}}\label{sec:chaikaex}

The following proposition shows that particular choices of $(x,y)$ need not satisfy Theorem \ref{thm:phi1mu}. 
It relies on a class of examples constructed in \cite{nonue} and its proof relies on 
results about a subclass of these examples in \cite{chaika1}. To prove Theorem \ref{thm:chaikaex} we establish the following proposition and Remark \ref{rem:renorm} finishes the argument.
\begin{prop} \label{prop:chaikaex}
There exists $T$, a minimal 4-IET with 2 ergodic measures, $\mu_{sing}, $ \leb \ \   
such that for $\mu_{sing} \times$ \leb-  a.e. pair $(x,y)$ 
$\underset{n \to \infty} {\liminf} \, n |T^n(x)-y|=\infty$.
\label{ex2}
\end{prop}
Compare it with Chebyshev's Theorem (\cite[Theorem 24]{Khinchin}):
\begin{thm} For an arbitrary irrational number $\alpha$ and real number $\beta$ the 
inequality $|n\alpha-m-\beta|<\frac 3 n$ has an infinite number of integer solutions $(n,m)$.
\end{thm}
In the language of this paper Chebyshev's Theorem states that if $T_{\alpha}$ is an 
irrational rotation, 
$\underset{n \to \infty} {\liminf} \, n |T_{\alpha}^n(x)-y|\leq 3$ for any $y$.

 The terminology that will be used is in \cite{chaika1}. Recalling a few key facts from 
 \cite{nonue} observe that by considering $I^{(k+1)}$ as the first induced map of $I^{(k)}$, 
 one can see that the images of $I_2^{(k+1)}$  in $O(I_2^{(k+1)})$ land $b_{k,4}$ times in 
 $O(I_4^{(k)})$, then $m_{k+1}b_{k,2}$ times in $O(I_2^{k})$, then $n_{k+1}b_{k,3}$ times 
 in $O(I_3^{(k)})$ before returning to $I^{(k+1)}$. (Recall that $O(I_j^{(k)})$ is the 
 Kakutani-Rokhlin tower over the $j^{\text{th}}$ interval of the $k^{\text{th}}$ induced 
 map. $b_{k,j}$ represents the number of images of $I_j ^{(k)}$ in this tower.) \\
We place the following conditions on $m_k$ and $n_k$.
\begin{enumerate}
\item $(n_k)^3<m_k $
\item $(b_{k-1,2})^2<m_k<(b_{k-1,2})^5$
\item $(b_{k,2})^2 2^{2k} m_k<n_{k+1}$.
\end{enumerate}
These conditions provide the following immediate consequences:
\begin{enumerate}
\item $b_{k,2}\geq b_{k,j}$ for any $j$ (Lemma 3 of \cite{chaika1}).
\item $(b_{k-1,2})^3<b_{k+1,2}<4(b_{k-1,2})^6$ (direct computation with condition 2).
\item \leb $(O(I_2^{k}))<\frac 2 {(b_{k,2})^2}$ (by Lemma 1 of \cite{chaika1} which is 
in the proof of Lemma 3 of \cite{nonue}).
\item $\underset{j=1}{\overset{\infty}{\sum}} \frac {n_{k+1}b_{k,3}}{b_{k+1,2}}$ 
converges ($b_{k+1,2}>m_{k+1}b_{k,2}$ and $b_{k,2}>b_{k,3}$).
\item $\underset{j=1}{\overset{\infty}{\sum}} \frac{n_k}{m_k}$ converges (condition 1).
\item $\mu_{sing}$ a.e. point is in $\underset{j=1}{\overset{\infty}{\cup}}\underset{k=j}
  {\overset{\infty}{\cap}} O(I_2^{(k)})$ (condition 1 and Lemma 6 of \cite{chaika1})
\end{enumerate}

\begin{defi} Let $A_{x,r,M,N}=\{y: |T^n(x)-y|< \frac r n \text{ for some } N <n \leq M \}$
\end{defi}
This example relies on showing that \leb ($A_{x,r,M,N}$) is small 
(see Lemma \ref{conv}) for a $\mu_{sing}$ large set of $x$.
The following definition provides us with a class of $x$ such that we can control 
\leb ($A_{x,r,M,N}$) as seen by Lemma \ref{conv}. 
This class is also $\mu_{sing}$ large as seen by Lemma \ref{good}.
\begin{defi} $x$ is called $\emph{k-good}$ if:
\begin{enumerate} 
\item $T^n(x) \in O(I_2^{(k)})$ for all $0\leq n \leq b_{k,2}$.
\item $T^n(x) \in O(I_2^{(k-1)})$ for all $0 \leq n \leq (n_{k} b_{k-1,3})^2$
\end{enumerate}
\end{defi}

Consequence (6) shows that $\mu_{sing} $ almost every point satisfies condition (1) for k-good for all $k>N$. The following lemma shows that condition (2) is also satisfied eventually.
\begin{lem} For $\mu_{sing}$ a.e. $x$ there exists $N$ such that $x$ is $k$-good for all $k>N$.
\label{good}
\end{lem}
\begin{proof} The basic reason $\mu_{sing}$ a.e. $x$ is eventually $k$-good for all large enough $k$ is that the images of $O(I_2^{(k+1)})$ not in  $O(I_2^{(k)})$ are consecutive (by the construction in \cite{nonue}; see the discussion immediately following the statement of Theorem \ref{ex2} in this paper).
 This means we need to avoid $n_{k+1}b_{k,3}+b_{k,4}+(n_k b_{k,3})^2$ image of $O(I_2^{(k+1)})$.
 Because
$$\frac {n_{k+1}b_{k,3}+b_{k,4}+(n_{k+1} b_{k,3})^2}{b_{k+1,2}}<\frac{(n_{k+1}+1)b_{k,2}}{m_{k+1}b_{k,2}}$$ 
 is a convergent sum, the Borel Cantelli Theorem implies $\mu_{sing}$ almost every $x$ is $k$-good for all big enough $k$. (The left hand side is the proportion of the images of $I_2^{(k+1)}$ which are not good.)
 \end{proof}
The next lemma shows that if $x$ is $k+1$ good then $ A_{x,r,b_{k,2},b_{k+1,2}}$ is small in terms of Lebesgue measure.

\begin{lem} If $x$ is $k+1$-good for all $k>N$ then \leb $ (A_{x,r,b_{k,2},b_{k+1,2}})$ forms a convergent sum.
\label{conv}
\end{lem}
\begin{proof} This proof will be carried out by estimating the measure $A_{x,r,b_{k,2},b_{k+1,2}}$ gains when $x$ lands in $O(I_2^{(k)})$ and when it doesn't.
Since $x$ is $k+1$-good, the Lebesgue measure $A_{x,r,b_{k,2},b_{k+1,2}}$ gains by not landing in $O(I_2^{(k)})$ is less than
$$\frac {2} {(n_{k+1} b_{k,3})^2} (n_{k+1}b_{k,3}+b_{k,4})  \leq \frac {2(n_{k+1}b_{k,3}+4b_{k-1,2})}{(n_{k+1} b_{k,3})^2} \leq \frac {4} {n_{k+1}b_{k,3}}.$$
 When $x$ lands in $O(I_2^{(k)})$ it either lands in one of the $b_{k-1,2}$ components of which are images of $I_2 ^{(k-1)}$  (this is $O(I_2^{(k-1)})$) or it doesn't.
 On each pass through of the orbit, it lands $m_kb_{k-1,2}$ in one of the first $b_{k-1,2} $ images of $I_2 ^{(k-1)}$, and $n_kb_{k-1,3}+b_{k-1,4}$ times it doesn't.
 We will estimate $ A_{x,r,b_{k,2},b_{k+1,2}}$ by dividing up the orbit into these pieces.\\
When $x$ lands in $O(I_2^{(k-1)})$ the measure of the points its landings place in $A_{x,r,b_{k,2},b_{k+1,2}}$ is at most \leb ($O(I_2^{(k-1)})$) $+\frac {2r} {b_{k,2}} b_{k-1,2}$. (There are $b_{k-1,2} $ connected components of $O(I_2^{(k-1)})$.)\\
Otherwise we approximate the measure by :
$$\frac 2 {b_{k,2}} \underset {i=1}{\overset{b_{k+1,2}/b_{k,2}}{\sum}} \frac{n_kb_{k-1,3}+b_{k-1,4}}{i} \leq \frac {n_kb_{k-1,3}+b_{k-1,4}}{b_{k,2}}\ln(b_{k+1,2})\leq \frac {n_kb_{k-1,3}+b_{k-1,4}}{b_{k,2}}7\ln(b_{k,2}).$$ 
The left hand side is given by estimating the measure gained by hits in $O(I_2^{(k)}) \backslash O(I_2^{(k-1)}$ ($n_kb_{k-1,3}+b_{k-1,4}$ hits each of which contributes at most $\frac 2 {ib_{k,2}}$ on the $i^{\text{th}}$ pass and summing over each pass through $O(I_2^{(k)})$). The first inequality is given by the fact that $b_{k+1,2} > \frac {b_{k+1,2}}{b_{k,2}}$. The final inequality is given 
by consequence 2.
 \\ Collecting all of the measure, if $x$ is $k+1$-good then
$$ \leb (A_{x,r,b_{k,2},b_{k+1,2}}) \leq \frac {4}{n_{k+1}b_{k,3}}+\frac 2 {(b_{k,2})^2}+\frac {2r} {b_{k,2}} b_{k-1,2}+\frac {n_kb_{k-1,3}+b_{k-1,4}}{b_{k,2}}7\ln(b_{k,2})$$

 which forms a convergent series due to the at least exponential growth of $b_{k,i}$.
\end{proof}

 \vspace{2mm}
 
 \begin{proof}[Proof of Proposition \ref{ex2}]  $\mu_{sing}$ a.e. $x$ is eventually $k+1$-good. By Borel-Cantelli for each of these $x$, Lebesgue a.e. $y$ has $\underset{n \to \infty}{\lim} \, n|T^n(x)-y|=\infty$. The set of all $(x,y)$ such that $\underset{n \to \infty}{\lim} \, n|T^n(x)-y|=\infty$ is measurable, and so has $\mu_{sing} \times \leb $ measure  1  (by Fubini's Theorem).
 \end{proof}
\begin{rem} One can modify conditions 1-3 to achieve $\liminf_{n\to\infty}\limits\,n^{\alpha}|T^nx-y|= \infty$ for $0<\alpha<1$.
\end{rem}
\begin{rem}\label{rem:renorm} Following \cite[Section 1]{ietv}, one can renormalize the IET by choosing the IET $S_p$, with length vector $$(p\leb(I_1)+(1-p)\mu_{sing}(I_1),p\leb(I_2)+(1-p)\mu_{sing}(I_2),p\leb(I_3)+(1-p)\mu_{sing}(I_3),p\leb(I_4)+(1-p)\mu_{sing}(I_4))$$
 and permutation $4213$. $S_p$ has the same symbolic dynamics and obeys the same \cite{nonue} type induction procedure as $T$ (with the same matrices). As a result $S$ has two ergodic measures $\mu_{S_p}$, $\lambda_{S_p}$ such that $\mu_{S_p}(I_j^{(k)})$ for $S_p$ is the same as $\mu_{sing}(I_j^{(k)})$ for $T$ and $\lambda_{S_p}(I_j^{(k)})$ for $S_p$ is the same as $\leb(I_j^{(k)})$. Moreover, if $0<p<1$ then $\mu_{S_p}$ and $\lambda_S$ are both absolutely continuous and supported on disjoint sets of Lebesgue measure $1-p$ and $p$ respectively. If $p=1$ the IET is $T$, if $p=0$ then $\mu_{S_0}$ is Lebesgue measure and $\lambda_{S_0}$ is singular. As a result the Lebesgue measure of $I_j^{(k)}$ for $S_{.5}$ is at least $.5 \max\{\leb (I_j^{(k)}) ,\mu_{sing}(I_j^{(k)})\}$ for $T$.  From this it follows that $\underset{n \to \infty} {\liminf} \, n |S_{.5}^n(x)-y|=\infty$ on a set of $(x,y)$ with measure .25 (corresponding to $x $ being chosen from a set of $\mu_S$ full measure  
 and $y$ being chosen from a set of $\lambda_S$). This proves Theorem \ref{thm:chaikaex}. (See \cite[Section 1]{ietv} for more on renormalizing).
\end{rem}

\section{\bf \large The $\tau$-entropy of an IET} \label{sec:tau}
Note that Lebesgue almost all  $3$-IET's (corresponding to  
$\pi=(3\ 2\ 1)\in S_3$)  are weakly mixing. This result  has been first 
proved by Katok and Stepin  \cite{KaS},  with an explicit sufficient  
diophantine  generic condition on the lengths $\{\ell_1,\ell_2,\ell_3\}$  
(of exchanged intervals) given.
A weaker sufficient generic condition (for weak mixing of a  $3$-IET)  
has been given in  \cite{BoN}.  

Next we introduce the notion of a  $\tau$-entropy, $\tau(T)$,  of an  IET  $T$                \mbc{9.14}
   which provides an upper bound  on 
the proximality constant   $C_\psi(T)$   (Theorem \ref{thm:psialp} below).

We consider the following finite subsets of  $\ui$  associated with 
a given  IET  $(\ui,T)$:                                                                      \mbc{9.15}    
\begin{subequations}\label{eqs:deltas}
\begin{align}
\Delta(T)&\bydef\{T(x)\ominus x  \mid x\in \ui\}\subset \ui;  \\ 
\Delta'(T)&\bydef\{x\ominus y\mid x,y\in\Delta(T)\}\subset \ui;  \\  
\Delta'_n(T)&\bydef \bigcup_{k=1}^n \Delta'(T^k), \quad \text{for } n\geq1;
\end{align}
\end{subequations}
where  $\oplus, \ominus$  stand for the binary operations on  $\ui$  defined as
addition and subtraction  modulo $1$.  
Clearly,  for an $r$-IET  $T$,   the following inequalities hold:  
$\card(\Delta(T))\leq r$,   $\card(\Delta(T^k))\leq rk$  (because  $T^k$
is an  IET  exchanging at most  $(r-1)k+1\leq rk$  subintervals),  
and hence  $\card(\Delta'_n(T))< r^2n^3$. 
\begin{def}\label{def:tauent}
The $\tau$-entropy of an  IET\, $T$  is defined by the formula
\begin{equation}\label{eq:tau}
\tau(T)=\limsup_{n\to\infty}\, \frac{\log(\card(\Delta'_n(T))}{\log n}.
\end{equation}
\end{def}
It is clear that the inequalities  $0\leq \tau(T)\leq 3$  hold  for any  IET\,  $T$.  
\begin{definition}
An  IET\,  $T$  is called  $\tau$-deterministic if   $\tau(T)=0$.
\end{definition}
It is clear that all  $2$-IETs (circle rotations)  are  $\tau$-deterministic
(because then  $\Delta'_n(T)=\{0\}$). It turns out that Lebesgue a.\,a.  
$3$-IETs  also are $\tau$-deterministic (see Theorem \ref{thm:om2tau}).
\begin{thm}\label{thm:psialp}  
If $T$ is a $\tau$-deterministic IET then its proximality constant $C_\psi(T)$ vanishes.    \mbc{9.2\\thm:psialp}
More generally, for an arbitrary  IET\, $T$,   $C_\psi(T)\leq \min(1,\tau(T))$.
\end{thm}
We will see that  certain assumption on the rate the orbits of an IET\, $T$  become 
uniformly  distributed  implies  that  $T$  is $\tau$-deterministic  
(see Theorem  \ref{thm:om2tau} below).

Theorem \ref{thm:psialp}  is a special case of the following proposition.                   
\begin{prop}\label{thm:psi}  
Let  $T$  be an IET.  Let  $\bfs=\{s_n\}_1^\infty$  be a scale sequence and let 
$\psi\colon \ui\times\ui\to[0,\infty]$
be an associated proximality gauge:  
$\psi(x,y)=\liminf_{n\to \infty}\limits\, s_n\,|T^n(x)-T^n(y)|$.
Set\,  $v_k=|\Delta'_{2k}(T)|/s_k$\,  and  assume that\,  
$\sum_{k\geq1}\limits  v_{2^k}<\infty$.
Then  $\psi(\lam,\lam)=\delta_{\infty}$, i.\,e. $\psi(x,y)=\infty \pmod{\lam\ttimes\lam}$.
\end{prop}
\begin{proof}[{\bf Proof of Proposition \ref{thm:psi}}]
Fix arbitrary  $x_0\in\ui$  and $t\in\R$, $t>0$.  
For  $k,n\geq1$,  denote
\[
Y_k=\big\{y\in\ui\,\big|\,\big. \|T^ky-T^kx_0\|\leq t/s_k\big\};   \qquad  
Z_n=\bigcup_{k=n}^{2n}\limits Y_k.
\]
Since \
$
T^ky\ominus T^kx_0=\big((T^ky\ominus y)\ominus(T^kx_0\ominus x_0)\big)
\oplus(y\ominus x_0)\subset (y\ominus x_0)\oplus\Delta'_k(T),
$
it follows that                                                                               \mbc{9.22}
\begin{equation}\label{yksub}
Y_k\subset x_0\oplus B(\Delta'(T^k), t/s_k)
\end{equation}
where\,
$B(\Delta,\eps)=\{x\in\ui\mid \dist(x,\Delta)<\eps\}$\,
denotes the $\eps$-neighborhood of a subset  $\Delta\subset\ui$,   
with the standard convention:   
$\dist(x,\Delta)\bydef  \inf_{z\in\Delta}\limits\|x-z\|$.

Since  $\{s_n\}$  is non-decreasing,  the following inclusions hold                          \mbc{9.23}
\begin{align*}
Z_n=\bigcup_{k=n}^{2n} Y_k&\subset 
x_0+\bigcup_{k=n}^{2n} B(\Delta'(T^k),t/s_k)\subset\\
&\subset x_0+B\Big(\bigcup_{k=n}^{2n} \Delta'(T^k),t/s_n\Big)
\subset  x_0+B(\Delta'_{2n}(T),t/s_n).
\end{align*}
Therefore\,   $\lam(Z_n)\leq \card(\Delta'_{2n}(T))\,2t/s_n\leq 2t\,v_n$.

By the Borel-Cantelli lemma,   
$\sum_{k\geq1}\limits  v_{2^k}<\infty\,
\implies
\lam(\limsup_{k\to\infty}\limits Z_{2^k})=0$.

It follows that   
$
\lam(\limsup_{n\to\infty}\limits Y_n)=0
$
(because  
$\limsup_{n\to\infty}\limits Y_n=\limsup_{n\to\infty}\limits Z_{n}$), 
thus\,  $\lam\left(\{y\in\ui\mid \psi(x_0,y)<t\}\right)=0$.                                 \mbc{9.24}

Since  $s>0$  is arbitrary,  we get\,  $\lam(\{y\in\ui\mid \psi(x_0,y)=\infty\})=1$.
The claim of the proposition follows from Fubini's Theorem since 
$x_0\in\ui$  is arbitrary.
\end{proof}
\begin{proof}[{\bf Proof of Theorem \ref{thm:psialp}}]
Assuming that  $\alpha>\tau$,  set  $\beta=(\alpha+\tau)/2>\tau,
\; \gamma=(\alpha-\tau)/2>0$.  Then, with notation as in 
Proposition  \ref{thm:psi}   with  $s_n=n^\alpha$,  we obtain
$
v_n=|\Delta'_{2n}(T)|/s_n\leq |\Delta'_{2n}(T)|/n^\alpha\leq  
n^\beta/n^\alpha=n^{-\gamma},
$
for all large  $n$.    It follows that  $\sum_{k\geq1}\limits v_{2^k}<\infty$
because  $v_{2^k}\leq 2^{-k\gamma}$  for large  $k$.

It remains to consider the case $\alpha>1$. Then, with  $s_n=n^\alp$ and notation 
as in the proof of  Proposition \ref{thm:psi}, we have  $\lam(Y_n)\leq tn^{-\alp}$.
By Borel-Cantelli lemma,  $\lam(\limsup_{n\to\infty}\limits Y_n)=0$  since  
$\sum_n\limits  tn^{-\alp}<\infty$.                                                          \mbc{9.25}

The completion of the proof is similar to the one of Proposition \ref{thm:psi}.  
\end{proof}
\vsp2

Theorem \ref{thm:psialp}  provides a motivation for finding estimates
for  $\tau$-entropies  $\tau(T)$,  for  various  IETs  $T$.  
We will show that  such estimates are possible in terms 
of the rate the orbits  of $T$  become uniformly distributed
(see Theorem \ref{thm:om2tau}).

\begin{def}\label{def:wdis}
Let  $T\colon\ui\to\ui$  be a map and $n\in\N$.  Set
\[
D_n(T)=\sup_{\substack{x\in\ui\\0<a<b<1}}\limits \Big| \frac1n\sum_{k=0}^{n-1} 
\theta_{a,b}(T^k(x))-(b-a)\Big|
\]
where\,  $\theta_{a,b}\colon\ui\to\{0,1\}$  stands for the characteristic function 
of the interval  $(a,b)\in\ui$.                                                            \mbc{9.26}  

By an\, $\omega$-discrepancy of\, $T$  we mean the following constant
\[
\omega(T)=\limsup_{n\to\infty}\limits\frac{\log\, (n\,D_n(T))}{\log n}=
1+\limsup_{n\to\infty}\frac{\log D_n(T)}{\log n}\in [0,1].
\]
\end{def}

Thus, for any\, $\eps>0$,  the inequality  $D_n(T)<n^{\omega(T)-1+\eps}$  
holds  for all\,  $n\in\N$\,  large enough.  

Note that an IET\, $T$  with  $\omega(T)<1$  must be minimal and uniquely 
ergodic  (because every orbit is uniformly distributed in~$\ui$).

The following lemma is an immediate consequence of the definition $D_n$.                 \mbc{9.27}  
\begin{lem}
Let  $T\colon\ui\to\ui$  be a map.  Then the following inequality holds for 
all\, $x,y,a,b\in\ui$  such that  $a<b$  and all\,  $n\in\N${\rm :}
\[
\Big| \sum_{k=0}^{n-1} \theta_{a,b}(T^k(x))-\sum_{k=0}^{n-1} 
\theta_{a,b}(T^k(y))\Big|\leq 2n\,D_n(T).
\]
In particular,  for $\eps>0$,
$
\Big| \sum_{k=0}^{n-1} \theta_{a,b}(T^k(x))-\sum_{k=0}^{n-1} 
\theta_{a,b}(T^k(y))\Big|\leq n^{\omega(T)+\eps},
$
for all  $n\in\N$  large enough.
\end{lem}
The notions of  $\omega$-discrepancy $\omega(T)$ and $\tau$-entropy $\tau(T)$ are 
motivated by the following result establishing an upper limit on the proximality
constant of an IET.                                                                       \mbc{9.28\\thm8\\om2tau}
\begin{thm}\label{thm:om2tau}
An  IET\,  $T$  with  $\omega(T)=0$  must be  $\tau$-deterministic,
i.e.\  $\tau(T)=C_\psi(T)=0$.  More generally, if\,  $T$  is an  
$r$-IET  with   $\omega(T)=\omega_0$,  then\,    
$C_\psi(T)\leq\tau(T)\leq (r-1)\,\omega_0$.
\end{thm}
\begin{proof}
In view of Theorem \ref{thm:psialp}, all we need is to validate the inequality
$\tau(T)\leq (r-1)\,\omega_0$.

Let\,  $\ui_1,\ui_2,\ldots,\ui_r$\,  be the subintervals exchanged by  $T$.              \mbc{9.29}
Then  $T(x)-x$  is constant on each  $\ui_k$.
Set  $h_k=T(x)-x$, for any  $x\in\ui_k$, $1\leq k\leq r$. 
It follows that  
\[   T^n(x)-x=\sum_{i=0}^{n-1} (T^{i+1}x-T^ix)=\sum_{k=1}^r  h_k\, H_{k,n}(x),
\]
where  $H_{k,n}(x)= \sum_{j=0}^{n-1}  \theta_{\ui_k}(T^jx)\in\Z$  and  
$\theta_{\ui_k}$ stand for the characteristic functions of intervals $\ui_k$.
 
It is clear that $\sum_{k=1}^r\limits H_{k,n}(x)=n$, for all  $x\in\ui$ and $n\in\N$.
 
Fix  $\eps>0$. Then, for large $n$,  
$|H_{k,n}(x)-n\lam(\ui_k)|\leq D_n(T)<n^{\omega_0+\eps}$,  whence                         \mbc{9.3}
$\Delta'(T^n)\subset  S_n$  where
\[    S_n\bydef \Big\{\sum_{k=1}^r  h_kH_k\,\Big|\,\big.  H_k\in \Z, 
      |H_k|< 2n^{\omega_0+\eps},\, \sum_{k=1}^r H_k=0\Big\}.
\] 
It follows that  $\Delta'_n(T)\subset S_n$ (see equations \eqref{eqs:deltas}),
and hence\,
$   \card(\Delta'_n(T))\leq \card(S_n)< (5n^{\omega_0+\eps})^{r-1}
$
for large  $n$  whence  $\tau(T)\leq (r-1)(\omega_0+\eps)$  (see \eqref{eq:tau}).  
The proof is complete since  $\eps>0$  is arbitrary.
\end{proof}
\begin{definition}
The type  $\nu=\nu(\alp)$  of an irrational number $\alp\in\rmq$\,
is defined as the limit
\[
\nu=\liminf_{n\to\infty}  \left(\frac{-\log\fracpart{n\alp}}{\log n}\right).
\]
\end{definition}

It is well known that  $\nu(\alp)\in[1,+\infty]$.  A number  $\alp\in\rmq$                \mbc{9.31}
is called {\em Liouville}  if   $\nu(\alp)=\infty$;  otherwise it is called 
{\em diophantine}.  Denote  
$\RR_s=\{\alp\in\rmq\mid  \nu(\alp)=s\}$.
The numbers of type 1 (i.e., the numbers in  $\RR_1$)  are also said 
to be of  {\em Roth type}.
It is easy to show that Lebesgue almost all numbers are of Roth type.

For references on the above definitions and related  basic results see  
\cite[Chapter 2, \S3, pp.121--125]{KN}  (in particular, see Lemma 3.1, 
Theorem 3.2  and Example 3.1). 
\begin{prop}
Let  $\alp\in\RR_s$,  for some  $1\leq s<\infty$.  Let  $T_\alp$  denote 
the  $2$-IET  corresponding to the  $\alp$-rotation on $\ui=[0,1)=S^1$
(i.e.,  $T(x)=x+\alp\!\pmod1$, $x\in\ui$.)  Then  $\omega(T)\leq\frac{s-1}s$.
In particular,  if  $\alp$  is of Roth type, then  $\omega(T)=0$.
\end{prop}
\begin{proof}
By  \cite[Chapter 2, Theorem 3.2]{KN}, for any  $\eps>0$, 
the inequality\, $D_n(T)<n^{\eps-1/s}$\,  holds  for large  $n$. 
The proposition follows from the definition of \ $\omega(T)$.
\end{proof}
\begin{proof}[Proof of Theorem \ref{thm:3psia}] This follows by direct modification 
of Theorem \ref{thm:om2tau} and the previous proposition. Almost every 3-IET is the 
induced map of a rotation by a Roth type $\alpha$ (i.\,e. $\alpha \in \RR_1$).
Fix $T$ to be a 3-IET given by the induced map of $R_{\alpha}$ on $[0,b)$ for 
$\alpha \in \RR_1$.  Thus $T(x)=(x \oplus \alpha)/b$\, if\, $x \in [0,b)$\, and\,
$(x \oplus 2\alpha)/b$\, otherwise. Therefore the growth of $\Delta'_n(T)$ is 
controlled by 
$\underset{x,y \in \ui}{\max} \underset{k=0}{\overset{2n}{\sum}} \theta_{b,1}R_{\alpha}(x)-
   \underset{k=0}{\overset{2n}{\sum}} \theta_{b,1}R_{\alpha}(y)$. 
Theorem \ref{thm:3psia} follows.
\end{proof}
\section{\bf \large No 3-IET is Topologically Mixing}
It is worth mentioning that at certain times the discrepancy of rotations is at most 4 valued \cite{kesten}. This provides the following theorem:
\begin{thm} \label{thm:3nomix}
No 3 IET is Topologically mixing.
\end{thm}
\begin{definition} $T:[0,1) \to [0,1)$ is called \emph{topologically mixing} if for any (nonempty) open sets $U$ and $ V$ there exists $N$ such that $T^n(U) \cap V \neq \emptyset$ for all $n>N$.
\end{definition}
It is sufficient to consider 3-IETs formed by inducing on irrational (non-periodic) rotations of the circle.
\begin{lem}\label{spec times} Let $R_{\alpha}$ be rotation of the circle($[0,1)$) by $\alpha$. Then for any interval $J \subset [0,1)$ the
cardinality of the set $\{x, x+ \alpha,...,(q_n-1) \alpha +x \} \cap J|$ takes at most 4 consecutive values as $x$ varies.
\end{lem}
Call the largest of these $b_n$
This Lemma is a consequence of Theorem 1 of \cite{kesten} which states:
\begin{thm}  \cite{kesten} Each interval $(\frac r {q_m}, \frac {r+1} {q_m}); r=0,1,...q_m -1$ contains exactly one point $k\alpha$ with $1 \leq k \leq q_m$.
\end{thm}
To prove Lemma \ref{spec times} observe that  $|\{x, x+ \alpha,...,(q_n-1) \alpha +x \} \cap J| \in [\lfloor \frac{|J|}{q_n} \rfloor -1, \lfloor \frac{|J|}{q_n} \rfloor +2] \cap \mathbb{N}$ 
\begin{lem} Let $T$ be a 3-IET obtained by inducing rotation by $\alpha$ on $[1-t,1)$. Then $T^{q_n -1 -b_n}(x)-x$ takes at most 7 consecutive values for each $n$.
\end{lem}
\begin{proof}[Proof of Theorem {\em \ref{thm:3nomix}}] Divide $[0,1)$ into intervals of length $\frac 1 {20}$. For each such interval, $J$ and each $n$ $T^{q_n -1 -b_n}(J)$ intersects at most fourteen of them. So it misses at least six. Therefore it misses at least one (in fact at least six) infinitely often, violating topological mixing.
\end{proof}
\section{\bf \large Questions}
We conclude by listing some open problems. 

{\bf Problem 1.} For  $r\geq4$  and permutations $\pi\in S_r$,  
    what are possible and what are ``typical"  values for  $C_\psi(T)$,  for  $r$-IETs\, $T$  
    corresponding to a given permutation $\pi$?
    
In fact, it is quite difficult (but possible) to construct an IET  $T$  with  $C_\psi(T)>0$,
even though we believe that for ``most'' $4$-IETs  $C_\psi(T)>0$. 

For ``most''  $3$-IETs  $C_\psi(T)=0$ (Theorem \ref{thm:3psia}),  and we conjecture that 
$C_\psi(T)=0$ for all $3$-IETs.

Note that the same question regarding the constants $C_\phi(T)$ and $C_\rho(T)$ 
is well understood.

{\bf Problem 2.} Is it possible to have a uniquely ergodic IET\, $T$\, with two points 
          $x, y$ such that \ $\psi_1(x,y)= \infty$\,?
          
If  $T$  is $\lam$-ergodic but not uniquely ergodic, then   $\psi_1(x,y)= \infty$\, may
be possible (Proposition \ref{prop:chaikaex}).

{\bf Problem 3.} Does there exists an $1$-collapsing IET  $T$?

$T$ is  $1$-collapsing  if \
$
\phi_1(x,y)=\liminf_{n\to\infty}\limits n|T^nx-y|=0,
$ \
for all $x,y\in\ui$.

\np

\end{document}